\documentclass[11pt]{amsart}

\usepackage[dvipsnames]{xcolor}
\usepackage{tikz}
\usetikzlibrary{positioning}
\usepackage{amsmath}
\usepackage{amssymb}
\usepackage[T1]{fontenc}
\usepackage[utf8]{inputenc}
\usepackage{mathtools}
\usepackage{amsfonts}
\usepackage{fancyhdr}
\usepackage{comment}
\usepackage{mathrsfs}
\usepackage{verbatim}
\usetikzlibrary{arrows}

\begin{document}

\title{A categorical reconstruction of crystals and quantum groups at $q=0$}
\author{Craig Smith}

\maketitle
\pagestyle{plain}

\begin{abstract}
The quantum co-ordinate algebra $A_{q}(\mathfrak{g})$ associated to a Kac-Moody Lie algebra $\mathfrak{g}$ forms a Hopf algebra whose comodules are precisely the $U_{q}(\mathfrak{g})$ modules in the BGG category $\mathcal{O}_{\mathfrak{g}}$. In this paper we investigate whether an analogous result is true when $q=0$. We classify crystal bases as coalgebras over a comonadic functor on the category of pointed sets and encode the monoidal structure of crystals into a bicomonadic structure. In doing this we prove that there is no coalgebra in the category of pointed sets whose comodules are equivalent to crystal bases. We then construct a bialgebra over $\mathbb{Z}$ whose based comodules are equivalent to crystals, which we conjecture is linked to Lusztig's quantum group at $v = \infty$.
\end{abstract}

\tableofcontents

\newtheorem{theorem}{Theorem}[section]
\newtheorem{corollary}[theorem]{Corollary}
\newtheorem{example}[theorem]{Example}
\newtheorem{lem}[theorem]{Lemma}
\newtheorem{obs}[theorem]{Observation}
\newtheorem{ass}[theorem]{Assumptions}
\newtheorem{prop}[theorem]{Proposition}
\theoremstyle{definition}
\newtheorem{defn}[theorem]{Definition}

\newtheorem{rem}[theorem]{Remark}
\numberwithin{equation}{section}

\newenvironment{definition}[1][Definition]{\begin{trivlist}
\item\hskip \labelsep {\bfseries #1}]}{\end{trivlist}}

\newenvironment{remark}[1][Remark]{\begin{trivlist}
\item[\hskip \labelsep {\bfseries #1}]}{\end{trivlist}}

\setcounter{section}{-1}

\section{Introduction}

The quantum co-ordinate algebra $A_{q}(\mathfrak{g})$ associated to a Kac-Moody Lie algebra $\mathfrak{g}$ forms a Hopf algebra whose comodules are precisely the $U_{q}(\mathfrak{g})$ modules in the BGG category $\mathcal{O}_{\mathfrak{g}}$. In this paper we investigate whether a similar result is true when $q=0$.\\

In Section 1 we recall some definitions and basic results about quantum groups and crystals.  This is followed by the construction of an algebra structure on the crystal base $\mathcal{B}$ associated to $A_{q}(\mathfrak{g})$ in Section 2, as well as a discussion of why a lack of rigidity in pointed sets means that we cannot immediately adapt the comultiplication on $A_{q}(\mathfrak{g})$ to the setting of crystals.\\

In Section 3 we take a more categorical approach. Using the Barr-Beck Theorem we classify crystal bases as coalgebras over a comonad $U$ on the category of pointed sets. This is done in as broad generality as possible before being applied to the category of crystals. We can then encode the monoidal structure of the category of crystals into a bicomonadic structure on $U$. In doing this we prove that there is no coalgebra in the category of pointed sets whose comodules are equivalent to crystal bases.\\

In Section 4 we endow the free abelian group $\mathbb{B}$ on the crystal base $\mathcal{B}$ of $A_{q}(\mathfrak{g})$ with the structure of a $\mathbb{Z}$-bialgebra analogous to that of $A_{q}(\mathfrak{g})$. We then prove that its based comodules are equivalent to crystals. In the case of $\mathfrak{sl}_{2}$ we give an explicit presentation of this bialgebra and compare this to a presentation of $A_{q}(\mathfrak{sl}_{2})$. This leads us to conjecture a relationship between this bialgebra and Lusztig's quantum group at $v = \infty$ in \cite{L3}.\\

\subsection*{Funding}
This work was supported by the Engineering and Physical Sciences Research Council [EP/M50659X/1].

\subsection*{Acknowledgements}
\thanks{The author would like to thank Kobi Kremnitzer for his expert supervision and continued support throughout this research, without which writing this paper would not have been possible. He would also like to thank Andr\'e Henriques and Kevin McGerty for their invaluable insights and advice.}

\section{Quantum groups and crystals}

\subsection{Quantum groups}

We begin by setting some notation and recalling some preliminary results. The following constructions of quantum groups can be seen in Kashiwara's paper \cite{K3} and in Jantzen's book \cite[p.~51]{J}, or for a more detailed account, see Lusztig's book \cite{L1}.
\\

\begin{defn}
Let $\mathfrak{g}$ be the Lie algebra defined by the data of
\begin{itemize}
\item[i)] a \emph{weight lattice} $\Phi$, a free $\mathbb{Z}$-module, with \emph{simple roots} $\alpha_{i} \in \Phi$, indexed by $i$ in some indexing set $I$, that form a basis of the \emph{root lattice} $\Psi$ (with respect to some Cartan subalgebra) contained in $\Phi$;
\item[ii)] a symmetric bilinear form $( \cdot , \cdot ) : \Phi \times \Phi \rightarrow \mathbb{Q}$ such that $( \alpha_{i}, \alpha_{i}) \in 2 \mathbb{N}$ and $( \alpha_{i}, \alpha_{j}) \leq 0$ for $i, j \in I, i \neq j$; and
\item[iii)] \emph{simple coroots} $\lambda_{i} \in \Phi^{\ast}=\text{Hom}_{\mathbb{Z}}(\Phi, \mathbb{Z})$ such that $\lambda_{i}(\alpha) = \frac{2(\alpha_{i}, \alpha)}{(\alpha_{i}, \alpha_{i})}$ for $i \in I, \alpha \in \Phi$.
\end{itemize}
Then $\mathfrak{g}$ is generated by elements $e_{i}, f_{i}, h_{i}$ for $i \in I$ subject to the relations
$$\begin{array}{cccc}
[h_{i},h_{j}] = 0, & [e_{i},f_{i}] = \delta_{ij} h_{i}, &
[h_{i},e_{j}] = \lambda_{i}(\alpha_{j})e_{j}, & [h_{i},f_{j}] = - \lambda_{i}(\alpha_{j})f_{j} \end{array},$$
and for $i \neq j$,
$$\begin{array}{cc}
(\text{ad}e_{i})^{1-\lambda_{i}(\alpha_{j})}e_{j} = 0, & (\text{ad}f_{i})^{1-\lambda_{i}(\alpha_{j})}f_{j} = 0, \end{array}$$
where ad is the \emph{adjoint map} $(\text{ad}x)(y) = [x,y]$.
\end{defn}

\begin{defn}
We will denote by
$$\Psi_{+}=\left\lbrace \sum_{i \in I} n_{i}\alpha_{i} \mid n_{i} \geq 0 \right\rbrace \subset \Psi$$
the \emph{positive roots}, and $\Psi_{-}=-\Psi_{+}$ the \emph{negative roots}. Let
$$\Phi_{+} = \{\alpha \in \Phi \mid \lambda_{i}(\alpha) \geq 0 \text{ for all } i \in I\}$$
be the \emph{dominant weights}. Then $\Phi$ has a partial ordering given by $\alpha \geq \beta$ if and only if $\alpha-\beta \in \Phi_{+}$.
\end{defn}

\begin{defn}
Take $q$ to be a nonzero element of our base field $k$ which is not a root of unity.  We may then define the \emph{quantised enveloping algebra} $U_{q}(\mathfrak{g})$ to be the algebra generated over our field $k$ by $e_{i}, f_{i}, q^{\lambda}$ for $i \in I, \lambda \in \Phi^{\ast}$, with the defining relations
$$\begin{array}{rl}
\text{for $\lambda=0$} & q^{\lambda}=1, \\
\text{for $\lambda_{1},\lambda_{2} \in \Phi^{\ast}$} & q^{\lambda_{1}}q^{\lambda_{2}}=q^{\lambda_{1} + \lambda_{2}},\\
\text{for $i \in I, \lambda \in \Phi^{\ast}$} & q^{\lambda} e_{i} q^{-\lambda} = q^{\lambda(\alpha_{i})}e_{i},\\
& q^{\lambda} f_{i} q^{-\lambda} = q^{- \lambda(\alpha_{i})}f_{i}, \\
& e_{i}f_{i} - f_{i}e_{i} = \delta_{ij} \frac{t_{i} - t_{i}^{-1}}{q_{i}-q_{i}^{-1}}\\
\text{for $i \neq j$} & \sum_{k=0}^{1-\lambda_{i}(\alpha_{j})} (-1)^{k} e_{i}^{(k)} e_{j} e_{i}^{(1-\lambda_{i}(\alpha_{j})-k)}\\
&= \sum_{k=0}^{1-\lambda_{i}(\alpha_{j})} (-1)^{k} f_{i}^{(k)} f_{j} f_{i}^{(1-\lambda_{i}(\alpha_{j})-k)} = 0
\end{array}$$
where $q_{i}=q^{\frac{(\alpha_{i},\alpha_{i})}{2}}$, $t_{i} = q^{\frac{(\alpha_{i}, \alpha_{i})}{2} \lambda_{i}}$, $[k]_{i}=\frac{q_{i}^{k} - q_{i}^{-k}}{q_{i} - q_{i}^{-1}}$, $[k]_{i}! = [1]_{i}[2]_{i}...[k]_{i}$, ${f_{i}^{(k)}=f_{i}^{k}/[k]_{i}!}$, and $e_{i}^{(k)}=e_{i}^{k}/[k]_{i}!$. Let us denote by $U_{q}(\mathfrak{n})$ (respectively $U_{q}(\mathfrak{n}^{-})$) the subalgebra of $U_{q}(\mathfrak{g})$ generated by $\{e_{i} \mid i \in I\}$ (respectively $\{f_{i} \mid i \in I\}$). Similarly let $U_{q}(\mathfrak{b})$ (respectively $U_{q}(\mathfrak{b}^{-})$) be the subalgebra of $U_{q}(\mathfrak{g})$ generated by $\{e_{i} \mid i \in I\} \cup \{q^{\lambda} \mid \lambda \in \Phi^{\ast}\}$ (respectively $\{f_{i} \mid i \in I\} \cup \{q^{\lambda} \mid \lambda \in \Phi^{\ast}\}$).
\end{defn}

\begin{example} If $\mathfrak{g}=\mathfrak{sl_2}$, $U_{q} ( \mathfrak{sl_2} )$ is the algebra generated by $e,f,t,t^{-1}$ with defining relations
$$tet^{-1} = q^{2}e, \, \, \, tft^{-1} = q^{-2}f, \, \, \, ef-fe = \frac{t-t^{-1}}{q-q^{-1}}.$$
For general $\mathfrak{g}$, the subalgebras of $U_{q}(\mathfrak{g})$ generated by $e_{i},f_{i},t_{i},t_{i}^{-1}$, denoted $U_{q}(\mathfrak{g})_{i}$, are isomorphic to $U_{q}(\mathfrak{sl_2})$.
\end{example}

\begin{defn}
We say that a left $U_{q}(\mathfrak{g})$ module $M$ is \emph{integrable} if $M$ decomposes into \emph{weight spaces} $M=\bigoplus_{\alpha \in \Phi}M_{\alpha}$,
$$M_{\alpha} = \{m \in M \mid q^{\lambda}m = q^{\lambda(\alpha)}m \text{ for all } \lambda \in \Phi^{\ast}\},$$
and for each $i \in I$, $M$ is a locally finite dimensional $U_{q}(\mathfrak{g})_{i}$ module. We then define $\mathcal{O}_{\mathfrak{g}}$ to be the category of integrable left $U_{q}(\mathfrak{g})$ modules that are locally finite dimensional as $U_{q}(\mathfrak{n})$ modules. Likewise we define integrable right $U_{q}(\mathfrak{g})$ modules, and an analogous category $\mathcal{O}_{\mathfrak{g}^{\text{op}}}$.
\end{defn}

\begin{prop}[Lusztig, \cite{L1}]
\label{CategoryOCompletelyReducible}
Representations in $\mathcal{O}_{\mathfrak{g}}$ are completely reducible, and all irreducible objects are, up to isomorphism, indexed by $\alpha \in \Phi^{+}$. These irreducibles, denoted $V(\alpha)$, can be expressed explicitly as the representation generated by a single vector $u_{\alpha}$, called the \emph{highest weight vector}, with the defining relations
$$e_{i}u_{\alpha} = 0= f_{i}^{1+\lambda_{i}(\alpha)}u_{\alpha}, \quad q^{\lambda}u_{\alpha}=q^{\lambda(\alpha)}u_{\alpha}, \quad \text{ for }i \in I, \lambda \in \Phi.$$
\end{prop}

\begin{example}
\label{sl2Example}
In the case of $\mathfrak{sl}_{2}$, these irreducible representations are $V(n)$ indexed by $n \in \mathbb{Z}_{\geq 0}$. They have a basis $B(n)=\{u_{i}^{(n)} \mid 0 \leq i \leq n\}$ of $t$-eigenvectors with
$$tu_{i}^{(n)} = q^{n-2i}u_{i}^{(n)}, \quad eu_{i}^{(n)} = [n-i+1]u_{i-1}^{(n)}, \quad fu_{i}^{(n)}=[i+1]u_{i+1}^{(n)}.$$
Here we use the notation $[n]=\frac{q^{n}-q^{-n}}{q-q^{-1}}$, $[0]=0$. So, up to scalar multiplication, $e$ decreases the index $i$, whilst $f$ increases the index. Thus we may define operators $\tilde{e}: u_{i}^{(n)} \mapsto u_{i-1}^{(n)}$ and $\tilde{f}: u_{i}^{(n)} \mapsto u_{i+1}^{(n)}$ on the basis $B(n)$. These are the \emph{Kashiwara operators}.
\end{example}

\begin{defn}
\label{AqgBialgebra}
Let $A_{q}(\mathfrak{g})$ denote the \emph{quantum co-ordinate algebra} defined as the direct sum $A_{q}(\mathfrak{g}) = \bigoplus_{\alpha \in \Phi_{+}} V(\alpha) \otimes V(\alpha)^{\ast}$, where $V(\alpha)^{\ast}$ denotes the dual vector space of $V(\alpha)$. The unit is $1 = v_{0} \otimes v_{0}^{\ast} \in V(0) \otimes V(0)^{\ast}$ and multiplication is defined by the composition
\begin{align*}
V(\alpha) \otimes V(\alpha)^{\ast} \otimes V(\beta) \otimes V(\beta)^{\ast} &\xrightarrow{\sim} V(\alpha) \otimes V(\beta) \otimes V(\beta)^{\ast} \otimes V(\alpha)^{\ast} \\
&\xrightarrow{\sim} V(\alpha) \otimes V(\beta) \otimes (V(\alpha) \otimes V(\beta))^{\ast}\\
& \rightarrow \left( \bigoplus_{\gamma} V(\gamma) \right) \otimes \left( \bigoplus_{\delta} V(\delta) \right)^{\ast} \\
& \twoheadrightarrow \bigoplus_{\gamma} V(\gamma) \otimes V(\gamma)^{\ast}
\end{align*}
where the third arrow is given by the decomposition into irreducible components and the fourth projects onto corresponding pairs of components. This algebra has a comultiplication given by
\begin{align*}
V(\alpha) \otimes V(\alpha)^{\ast} &\cong V(\alpha) \otimes k \otimes V(\alpha)^{\ast} \\
&\rightarrow V(\alpha) \otimes (V(\alpha)^{\ast} \otimes V(\alpha)) \otimes V(\alpha)^{\ast} \hookrightarrow A_{q}(\mathfrak{g}) \otimes A_{q}(\mathfrak{g})
\end{align*}
induced by the coevaluation maps $k \rightarrow V(\alpha)^{\ast} \otimes V(\alpha)$, and counit given by the evaluation maps $V(\alpha) \otimes V(\alpha)^{\ast} \rightarrow k$.
\end{defn}

\begin{remark}
By the quantum Peter-Weyl Theorem (Proposition 7.2.2 of \cite{K5}), this can be identified with a subalgebra of functions on the quantum enveloping algebra $U_{q}(\mathfrak{g})$, where $u \otimes v \in V(\alpha) \otimes V(\alpha)^{\ast}$ is seen as the function $x \mapsto \langle x \cdot u, v \rangle$. The image of $A_{q}(\mathfrak{g})$ is then the subalgebra of all functions in $U_{q}(\mathfrak{g})^{\ast}$ such that the left and right $U_{q}(\mathfrak{g})$ submodules of $U_{q}(\mathfrak{g})^{\ast}$ they generate are in $\mathcal{O}_{\mathfrak{g}}$ and $\mathcal{O}_{\mathfrak{g}^{\text{op}}}$ respectively.
\end{remark}

\subsection{The category of crystals}
We now describe the category of crystals, a generalisation of crystal bases, as Kashiwara defines in \cite{K3}. See \emph{loc. cit.} for the motivation and intuition behind the following definitions.

\begin{defn}
A \emph{pointed set} is a set with a distinct point or element, which we shall always denote by $0$. A morphism between pointed sets is a map of sets which preserves $0$. We denote this category $\mathit{Set}_{\bullet}$. We give this category the monoidal structure
$$A \otimes B := \{(a,b) \in A \times B \mid a \neq 0 \neq b\} \sqcup \{0\}.$$
We usually denote by $a \otimes b$ the nonzero elements $(a,b)$ in $A \otimes B$.
\end{defn}

\begin{defn}
The \emph{category of crystals}, denoted $\mathit{Crys}$, has as objects pointed sets $B$ equipped with maps
$$\tilde{e}_{i}:B  \rightarrow  B, \quad \tilde{f}_{i}:B  \rightarrow  B, \quad \text{wt}: B  \rightarrow  \Phi,$$
for all $i \in I$ such that, for a crystal $B$ and $b,b_{1},b_{2} \in B$,
\begin{itemize}
\item[i)] if $\tilde{e}_{i}(b) \neq 0$ then $\text{wt} (\tilde{e}_{i}b) = \text{wt} (b) + \alpha_{i}$;
\item[ii)] if $\tilde{f}_{i}(b) \neq 0$ then $\text{wt} (\tilde{f}_{i}b) = \text{wt} (b) - \alpha_{i}$; and
\item[iii)] $b_{2} = \tilde{f}_{i} b_{1}$ if and only if $b_{1} = \tilde{e}_{i} b_{2}$.
\end{itemize}
We will call $\tilde{e}_{i}$ and $\tilde{f}_{i}$ the \emph{Kashiwara operators}. For crystals $B_{1}, B_{2}$, we say that a map $\psi : B_{1} \rightarrow B_{2}$ is a morphism of crystals if, for all $b \in B_{1}$ and $i \in I$, $\psi(\tilde{e}_{i}b)=\tilde{e}_{i}\psi(b)$ and $\psi(\tilde{f}_{i}b)=\tilde{f}_{i}\psi(b)$. We will denote by $\varepsilon$ and $\phi$ the functions $B \setminus \{0\} \rightarrow \mathbb{Z} \sqcup \{\infty\}$ given by $\varepsilon_{i}(b)=\text{max}\{n \geq 0 \mid \tilde{e}_{i}^{n}(b) \neq 0\}$ and $\phi_{i}(b)=\text{max}\{n \geq 0 \mid \tilde{f}_{i}^{n}(b) \neq 0\}$ for $b \in B \setminus \{0\}$.
\end{defn}

\begin{remark}
Some definitions of crystals, for example in \cite{K3}, have a broader definition of morphism and use the term \emph{strict} to distinguish the morphisms we have defined above. 
\end{remark}

\begin{defn}
We will call a crystal \emph{finite} if its underlying pointed set is of finite cardinality. A pointed subset of a crystal $B$ is a \emph{subcrystal} if it is closed under the action of $\tilde{e}_{i}$ and $\tilde{f}_{i}$ for all $i \in I$. We say that a crystal is \emph{irreducible} if it has no nontrivial proper subcrystals.
\end{defn}

The main source of examples of objects in $\mathit{Crys}$ are crystal bases of integrable $U_{q}(\mathfrak{g})$-modules. We omit the rather involved definition of a crystal base and instead refer interested readers to \cite{K3}. What is important to note is that part of the data of a crystal base of a $U_{q}(\mathfrak{g})$-modules $M$ is a pointed subset, $B\subset M$. The subset $B$ is closed under the action of Kashiwara operators $\tilde{e}_{i}$ and $\tilde{f}_{i}$ for all $i \in I$. Furthermore, each element $b \in B$ is homogeneous with respect to the weight space decomposition, hence has an associated weight, $\text{wt}(b)$. Thus a crystal base gives a crystal, $B$, in $\mathit{Crys}$. This can be seen for $\mathfrak{g}=\mathfrak{sl}_{2}$ in Example \ref{sl2Example}.

\begin{theorem} (Kashiwara \cite{K6})
Each $V(\alpha)$ has a unique crystal base, up to equivalence, with associated crystal $B(\alpha)$ such that $B(\alpha) \cap V(\alpha)_{\alpha} = \{u_{\alpha}\}$. Furthermore,
$$B(\alpha)=\{\tilde{f}_{i_{1}}^{n_{1}}\tilde{f}_{i_{2}}^{n_{2}} ..\tilde{f}_{i_{k}}^{n_{k}}u_{\alpha} \mid i_{1},i_{2},..,i_{k} \in I, n_{1},n_{2},..,n_{k} \geq 0 \}.$$
\end{theorem}

\begin{remark}
By Proposition \ref{CategoryOCompletelyReducible}, any integrable $U_q(\mathfrak{g})$-module in $\mathcal{O}_{\mathfrak{g}}$ has a unique crystal in $\mathit{Crys}$ arising as a disjoint union of crystals of the form $B(\alpha)$.
\end{remark}

\begin{defn}
We shall call crystals which are coproducts of crystals of the form $B(\alpha)$, as described in the previous remark, the \emph{crystals arising from integrable $U_q(\mathfrak{g})$-modules}. We shall denote their full subcategory $\mathit{Crys}_{\mathfrak{g}}$.
\end{defn}

\begin{defn}
For a crystal $B$, we define the crystal $B^{\vee} := \{ b^{\vee} \mid b \in B \}$ such that $\tilde{e}_{i}(b^{\vee}) = (\tilde{f}_{i}b)^{\vee}$,  $\tilde{f}_{i}(b^{\vee}) = (\tilde{e}_{i}b)^{\vee}$ and $\text{wt}(b^{\vee}) = -\text{wt}(b)$. For $\alpha \in \Phi_{+}$ we will denote $B(-\alpha) := B(\alpha)^{\vee}$.
\end{defn}

\begin{defn}
For a crystal $B$ we define its \emph{crystal graph} to be the graph whose vertices are the nonzero points in $B$ with arrows labeled by $i \in I$, $b\overset{i}{\longrightarrow }b'$ if and only if $b'=\tilde{f}_{i}b$.
\end{defn}

\begin{remark}
Crystal graphs are made up of disjoint unions of connected components. Since each subgraph of a crystal graph gives a subcrystal of the corresponding crystal, a crystal base is irreducible if and only if its crystal graph is connected.
\end{remark}

\begin{example} If $\mathfrak{g} = \mathfrak{sl_2}$, each irreducible $U_q(\mathfrak{sl_2})$-module $V(n)$, $n \in \mathbb{Z}_{\geq 0}$, has a corresponding crystal $B(n):=\{ u _{k} ^{(n)} \} _{0 \leq k \leq n}$. This has crystal structure defined by
$$\tilde{f}(u _{k} ^{(n)})=
\begin{cases} 
 u _{k+1} ^{(n)}, & k<n, \\
 0 & k=n,
\end{cases}
 \quad \tilde{e}(u _{k} ^{(n)})=
\begin{cases} 
 u _{k-1} ^{(n)}, & k>0, \\
 0 & k=0,
\end{cases}$$
so that
$$\varepsilon (x^iy^{n-i}) =  i, \quad \phi (x^iy^{n-i}) =  n-i, \quad \text{wt} (x^iy^{n-i}) = n-2i.$$
So the crystal base of an irreducible $U_{q} ( \mathfrak{sl_2} )$-module would have crystal graph
$$\circ \overbrace{ \rightarrow \circ \rightarrow \circ ..... \circ \rightarrow \circ \rightarrow}^{\varepsilon (b)} b \overbrace{ \rightarrow \circ \rightarrow \circ ..... \circ \rightarrow \circ \rightarrow}^{\phi (b)} \circ.$$
Thus we have $B(-n) \cong B(n)$ for $n \in \mathbb{Z}_{\geq 0}$.
\end{example}

The following important result characterises morphisms between irreducible crystals.

\begin{lem}[Schur's Lemma for Strict Morphisms,\cite{HK}]
\label{SchursLemma}
A nonzero morphism between irreducible crystals in $\mathit{Crys}_{\mathfrak{g}}$ is an isomorphism.
\end{lem}

Kashiwara defines the following monoidal structure on $\mathit{Crys}$ in \cite{K7}.

\begin{defn}
\label{CrystalTensor}
Let $B_{1}$, $B_{2}$ be crystals. Their tensor product is the pointed set $B_{1} \otimes B_{2} = \{b_{1} \otimes b_{2} \mid b_{1} \in B_{1}, b_{2} \in B_{2}\}$ with
$$\begin{array}{rcl}
\tilde{e}_{i}(b_{1} \otimes b_{2}) & = & \begin{cases}
       \tilde{e}_{i}b_{1} \otimes b_{2}  &\quad \text{if } \phi_{i} (b_{1}) \geq \varepsilon_{i} (b_{2})\\
       b_{1} \otimes \tilde{e}_{i}b_{2}  &\quad \text{if } \phi_{i} (b_{1}) < \varepsilon_{i} (b_{2}),\\
     \end{cases} \\
\tilde{f}_{i}(b_{1} \otimes b_{2}) & = & \begin{cases}
       \tilde{f}_{i}b_{1} \otimes b_{2}  &\quad \text{if } \phi_{i} (b_{1}) > \varepsilon_{i} (b_{2})\\
       b_{1} \otimes \tilde{e}_{i}b_{2}  &\quad \text{if } \phi_{i} (b_{1}) \leq \varepsilon_{i} (b_{2}),\\
     \end{cases} \\
\text{wt}(b_{1} \otimes b_{2}) & = & \text{wt}(b_{1}) + \text{wt}(b_{2}).
\end{array}$$
The unit for this monoidal structure is the crystal $\mathbb{B}(0)=\{b_{0},0\}$ where $\tilde{e}_{i}b_{0}=0=\tilde{f}_{i}b_{0}$ and $\text{wt}(b_{0})=0$.
\end{defn}

\begin{example}
\label{CrystalTensorsl2}
In the $\mathfrak{sl_2}$ case, for the crystals $B(n), B(m)$ this can be visualised as follows:

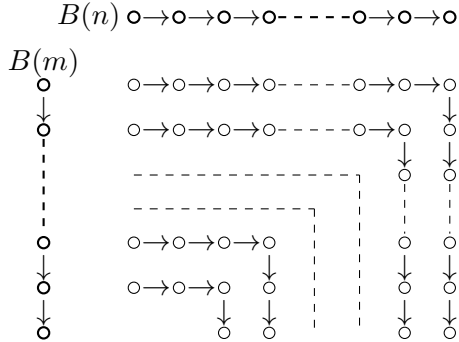
\begin{tikzpicture}[scale=0.3]

\node at (-3,0) {$B(m)$}; \node at (-1,2) {$B(n)$};

\draw(1,-1) circle [radius=0.25];\draw(3,-1) circle [radius=0.25];\draw(5,-1) circle [radius=0.25];
\draw(7,-1) circle [radius=0.25];\draw(11,-1) circle [radius=0.25];\draw(13,-1) circle [radius=0.25];
\draw(15,-1) circle [radius=0.25];
\draw(1,-3) circle [radius=0.25];\draw(3,-3) circle [radius=0.25];\draw(5,-3) circle [radius=0.25];
\draw(7,-3) circle [radius=0.25];\draw(11,-3) circle [radius=0.25];\draw(13,-3) circle [radius=0.25];
\draw(15,-3) circle [radius=0.25];
\draw(13,-5) circle [radius=0.25];\draw(15,-5) circle [radius=0.25];
\draw(1,-8) circle [radius=0.25];\draw(3,-8) circle [radius=0.25];\draw(5,-8) circle [radius=0.25];
\draw(7,-8) circle [radius=0.25];\draw(13,-8) circle [radius=0.25];\draw(15,-8) circle [radius=0.25];
\draw(1,-10) circle [radius=0.25];\draw(3,-10) circle [radius=0.25];\draw(5,-10) circle [radius=0.25];
\draw(7,-10) circle [radius=0.25];\draw(13,-10) circle [radius=0.25];\draw(15,-10) circle [radius=0.25];
\draw(5,-12) circle [radius=0.25];\draw(7,-12) circle [radius=0.25];\draw(13,-12) circle [radius=0.25];\draw(15,-12) circle [radius=0.25];

\node at (2,-1.1) {$\rightarrow$};\node at (4,-1.1) {$\rightarrow$};\node at (6,-1.1) {$\rightarrow$};
\node at (12,-1.1) {$\rightarrow$};\node at (14,-1.1) {$\rightarrow$};
\node at (2,-3.1) {$\rightarrow$};\node at (4,-3.1) {$\rightarrow$};\node at (6,-3.1) {$\rightarrow$};
\node at (12,-3.1) {$\rightarrow$};
\node at (2,-8.1) {$\rightarrow$};\node at (4,-8.1) {$\rightarrow$};\node at (6,-8.1) {$\rightarrow$};
\node at (2,-10.1) {$\rightarrow$};\node at (4,-10.1) {$\rightarrow$};

\node at (15,-2.1) {$\downarrow$};\node at (13,-4.1) {$\downarrow$};\node at (15,-4.1) {$\downarrow$};
\node at (7,-9.1) {$\downarrow$};\node at (13,-9.1) {$\downarrow$};\node at (15,-9.1) {$\downarrow$};
\node at (5,-11.1) {$\downarrow$};\node at (7,-11.1) {$\downarrow$};\node at (13,-11.1) {$\downarrow$};
\node at (15,-11.1) {$\downarrow$};

\draw [dashed] (7.4, -1) -- (10.6,-1); \draw [dashed] (7.4, -3) -- (10.6,-3);
\draw [dashed] (1, -5) -- (11,-5); \draw [dashed] (1, -6.5) -- (9,-6.5);
\draw [dashed] (11, -5) -- (11,-12); \draw [dashed] (9, -6.5) -- (9,-12);
\draw [dashed] (13, -5.4) -- (13,-7.6);\draw [dashed] (15, -5.4) -- (15,-7.6);

\draw[thick](1,2) circle [radius=0.25];\draw[thick](3,2) circle [radius=0.25];\draw[thick](5,2) circle [radius=0.25];
\draw[thick](7,2) circle [radius=0.25];\draw[thick](11,2) circle [radius=0.25];\draw[thick](13,2) circle [radius=0.25];
\draw[thick](15,2) circle [radius=0.25];

\node at (2,1.9) {$\rightarrow$};\node at (4,1.9) {$\rightarrow$};\node at (6,1.9) {$\rightarrow$};
\node at (12,1.9) {$\rightarrow$};\node at (14,1.9) {$\rightarrow$};

\draw [dashed, thick] (7.4, 2) -- (10.6,2);

\draw[thick](-3,-1) circle [radius=0.25];\draw[thick](-3,-3) circle [radius=0.25];\draw[thick](-3,-8) circle [radius=0.25];
\draw[thick](-3,-10) circle [radius=0.25];\draw[thick](-3,-12) circle [radius=0.25];

\node at (-3,-2.1) {$\downarrow$};\node at (-3,-9.1) {$\downarrow$};\node at (-3,-11.1) {$\downarrow$};

\draw [dashed, thick] (-3, -3.4) -- (-3,-7.6);

\draw(-14,-1) circle [radius=0];

\end{tikzpicture}
\end{example}

\begin{remark}
In Henriques and Kamnitzer's paper on \emph{Crystals and Coboundary Categories} \cite{HK} they describe a \emph{commuter for crystals} $\sigma_{B_{1} \otimes B_{2}} : B_{1} \otimes B_{2} \rightarrow B_{2} \otimes B_{1}$ for crystals $B_1,B_2$. This provides a way of commuting tensor products of crystals. If $\zeta :B \rightarrow B$ exchanges the highest and lowest weight elements of a crystal, essentially reversing the crystal graph, then $b\otimes b' \mapsto \zeta(\zeta(b')\otimes \zeta(b))$.
\end{remark}

\begin{prop}[\cite{K3}]
\label{DecompositionOfCrystals}
For $\alpha, \beta \in \Phi_{+}$, there is an isomorphism of crystals
$$B(\alpha) \otimes B(\beta) \cong \bigsqcup B(\alpha + \text{wt}(b))$$
where the disjoint union ranges over all $b \in B(\beta)$ such that $\varepsilon_{i}(b) \leq \lambda_{i}(\alpha)$ for all $i \in I$.
\end{prop}

\begin{corollary}
For $\alpha, \beta \in \Phi_{+}$, $B(\alpha + \beta)$ appears as a term in the decomposition of $B(\alpha) \otimes B(\beta)$ into irreducible components.
\end{corollary}

\begin{proof}
This follows from Proposition \ref{DecompositionOfCrystals}, since $\varepsilon_{i} (u_{\beta}) = 0 \leq \lambda_{i}(\alpha)$ for each $i$.
\end{proof}

\section{The crystal algebra $\mathcal{B}$}

\subsection{The crystal associated to $A_{q}(\mathfrak{g})$}

Recall the definition of the quantum co-ordinate ring $A_{q}(\mathfrak{g})$ from the first section. It is known that its comodules are precisely the representations of $U_{q}(\mathfrak{g})$ in $\mathcal{O}_{\mathfrak{g}}$. This is because, as a coalgebra, $A_{q}(\mathfrak{g})$ is a direct sum of coalgebras $V(\alpha) \otimes V(\alpha)^{\ast}$ whose comodules are precisely direct sums of copies of $V(\alpha)$. The focus of this section is to investigate whether an analogous result is true in the setting of crystal bases. We will consider the corresponding crystal
$$\mathcal{B} := \bigsqcup\nolimits _{\alpha \in \Phi^{+}} B(\alpha) \otimes B(-\alpha).$$

\begin{defn}
\label{Multiplication on crystal B}
Let $\mathcal{B}$ be the crystal $\mathcal{B}:=\bigsqcup_{\alpha \in \Phi_+} B(\alpha) \otimes B(-\alpha)$. For $\alpha, \beta \in \Phi$, we will denote by $b \cdot b'$ the image of ${b \otimes b'}$ in $B(\alpha) \otimes B(\beta)$ under the decomposition into irreducible components
$$B(\alpha) \otimes B(\beta) \cong \bigsqcup\nolimits_{\gamma \in \Gamma_{\alpha,\beta}} B(\gamma).$$
We then define a map
$$\mu_{\alpha,\beta}:B(\alpha) \otimes B(-\alpha) \otimes B(\beta) \otimes B(-\beta) \rightarrow \bigsqcup_{\gamma \in \Gamma_{\alpha,\beta}}B(\gamma) \otimes B(-\gamma)$$
by mapping $b \otimes b'^{\vee} \otimes d \otimes d'^{\vee}$ to $(b\cdot d) \otimes (b' \cdot d')^{\vee}$ whenever $b \otimes d$ and $b' \otimes d'$ lie in the same irreducible component of $B(\alpha) \otimes B(\beta)$ and to $0$ if not. Collectively, these induce a map $\mathcal{B} \otimes \mathcal{B} \rightarrow \mathcal{B}$, which we will denote $\mu$. Let $\eta$ denote the embedding
$$\eta: B(0) \cong B(0) \otimes B(0) \hookrightarrow \mathcal{B}.$$
\end{defn}

Note that the maps $\mu_{\alpha,\beta}$ above are not morphisms of crystals, just of pointed sets.

\begin{prop}
\label{Crystal Algebra B}
The maps $\mu$ and $\eta$ define an algebra structure on $\mathcal{B}$ in $\text{Set}_{\bullet}$.
\end{prop}

\begin{proof}
Let $\alpha, \beta, \gamma \in \Phi_{+}$ and let $b,b' \in B(\alpha)$, $c,c' \in B(\beta)$, $d,d' \in B(\gamma)$. By the associativity of the monoidal structure in Definition \ref{CrystalTensor}, $(b \cdot c) \cdot d$ corresponds to $b \cdot (c \cdot d)$ under the decomposition of $(B(\alpha) \otimes B(\beta)) \otimes B(\gamma) \cong B(\alpha) \otimes (B(\beta) \otimes B(\gamma))$ into irreducible components. Thus both $(\mu \otimes \text{Id}) \circ \mu$ and $(\text{Id} \otimes \mu) \circ \mu$ map $b \otimes b'^{\vee} \otimes c \otimes c'^{\vee} \otimes d \otimes d'^{\vee}$ to $(b \cdot c \cdot d) \otimes (b' \cdot c' \cdot d')^{\vee}$ if $b \otimes c \otimes d$ and $b' \otimes c' \otimes d'$ lie in the same irreducible component of $B(\alpha) \otimes B(\beta) \otimes B(\gamma)$ or $0$ otherwise. So $\mu$ is an associative multiplication. Since $B(0) \otimes B(\alpha) \xrightarrow{\sim} B(\alpha)$, $b_{0} \otimes x \mapsto x$, is an isomorphism, and likewise $B(\alpha) \otimes B(0) \xrightarrow{\sim} B(\alpha)$, $\eta$ is a unit for this multiplication.
\end{proof}

\begin{defn}
For a symmetric monoidal category $\mathcal{C}$ with monoidal unit $\mathbb{I}$, we say that an object $A^{\vee}$ is \emph{dual} to an object $A$ in $\mathcal{C}$ if there exist maps
$$\iota_{A}: \mathbb{I} \rightarrow A^{\vee} \otimes A, \quad \epsilon_{A}:A \otimes A^{\vee} \rightarrow \mathbb{I},$$
called the coevaluation and evaluation respectively, such that the composition
$$A \cong \mathbb{I} \otimes A \xrightarrow{\text{Id} \otimes \iota_{A}} A \otimes A^{\vee} \otimes A \xrightarrow{\epsilon_{A} \otimes \text{Id}}A \otimes \mathbb{I} \cong A$$
is the identity on $A$. We will say that $A$ is \emph{dualisable} if such a dual $A^{\vee}$ exists.
\end{defn}

Recall that, in Definition \ref{AqgBialgebra}, the comultiplication of $A_{q}(\mathfrak{g})$ is induced by coevaluation maps $k \rightarrow V(\alpha)^{\ast} \otimes V(\alpha)$. These exists since each $V(\alpha)$ is dualisable in the category of vector spaces, with dual $V(\alpha)^{\ast}$. We do not, however, have dualisability for $B(\alpha)$ in $\mathit{Set}_{\bullet}$.

\begin{lem}
The pointed set $B(\alpha)$ is not dualisable in the symmetric monoidal category $\mathit{Set}_{\bullet}$ for nonzero $\alpha \in \Phi_{+}$.
\end{lem}

\begin{proof}
Suppose we have a pointed set $A$ along with evaluation and coevalutation maps $\epsilon$ and $\iota$ that exhibit $A$ as a dual to $B(\alpha)$. The monoidal unit in $\mathit{Set}_{\bullet}$ is the pointed set $\mathbb{I}=\{1,0\}$, and so the map $\iota$ is given by an element $\iota(1)=a \otimes b \in A \otimes B(\alpha)$. Then for any $b' \in B(\alpha)$
$$1 \otimes b'= \epsilon(a \otimes b) \otimes b \in \mathbb{I} \otimes B(\alpha)$$
so $b=b'$. But this gives a contradiction as $B(\alpha)$ has more than one non-zero element for $\alpha \neq 0$.
\end{proof}

The above lemma means we cannot proceed in direct analogy to $A_{q}(\mathfrak{g})$ to give $\mathcal{B}$ a bialgebra structure. In Section \ref{WorkingOverZ} we will work in $\mathbb{Z}$-modules instead of pointed sets where we regain dualisability and can construct a bialgebra structure on $\mathbb{B}:=\mathbb{Z}\mathcal{B}$. Before that, we will use a categorical approach to determine that $\mathit{Crys}_{\mathfrak{g}}$ cannot be reconstructed as comodules over a coalgebra in $\mathit{Set}_{\bullet}$ but can be reconstructed as coalgebras over a comonad on this category.

\section{A functorial approach to crystals}

\subsection{Comonads and the Barr-Beck Theorem}

We recalling the definitions of comonads, which generalised notions of coalgebras in the setting of functors on categories. For more details see Borceaux's \emph{Handbook of Categorical Algebra 2} \cite[p.~189-197]{B}.

\begin{defn}
A \emph{comonad} on a category $\mathcal{C}$ is a triple $\mathbb{U}=(U, \varepsilon, \Delta)$, where $U: \mathcal{C} \rightarrow \mathcal{C}$ is a functor, and $\varepsilon : U \Rightarrow \text{id}_{\mathcal{C}}$ and $\Delta : U \Rightarrow U \circ U$ are natural transformations satisfying
$$(\text{Id} \ast \Delta)\circ \Delta = (\Delta \ast \text{Id})\circ \Delta: U \Rightarrow U \circ U \circ U$$
and
$$(\text{Id} \ast \varepsilon)\circ \Delta = \text{Id} = (\varepsilon \ast \text{Id})\circ \Delta: U \Rightarrow U.$$
Here, $\ast$ denotes the horisontal composition of natural transformations. A \emph{coalgebra} on this monad, which corresponds to the idea of a comodule over a traditional coalgebra, is a pair $(D, \zeta)$ where $D$ is an object in the category and $\zeta : D \rightarrow U(D)$ is a morphism in $\mathcal{C}$ satisfying
$$U(\zeta) \circ \zeta = \Delta_{D} \circ \zeta : D \rightarrow UU(D) \text{ and }\varepsilon_{D} \circ \zeta=\text{Id}_{D}.$$
A \emph{morphism of coalgebras} $g:(D, \zeta) \rightarrow (D', \zeta')$ is a morphism $g:D \rightarrow D'$ in the category such that $U(g) \circ \zeta = \zeta' \circ g$. Coalgebras in $\mathcal{C}$ over a comonad $\mathbb{U}$ form a category which we shall denote $\mathcal{C}_{\mathbb{U}}$.
\end{defn}

\begin{remark}
Suppose we have an adjunction $F: \mathcal{C} \leftrightarrow \mathcal{D}:G$ with unit $\eta : \text{id}_{\mathcal{C}} \Rightarrow G\circ F$ and counit $\varepsilon : F \circ G \Rightarrow \text{id}_{\mathcal{D}}$. Then $\mathbb{U} = (U := F\circ G, \varepsilon, \Delta)$ forms a comonad where $\Delta := \text{id}_{F} \ast \eta \ast \text{id}_{G}$. Furthermore, we have a \emph{comparison functor} $J_{\mathbb{U}} : \mathcal{C} \rightarrow \mathcal{D}_{\mathbb{U}}$ defined by
$$J_{\mathbb{U}}(A)  =  (F(A),F(\eta_{A})), \quad J_{\mathbb{U}}(f)  =  F(f)$$
for all objects $A$ and morphisms $f$ in $\mathcal{C}$.
\end{remark}

\begin{defn}
A functor $F: \mathcal{C} \rightarrow \mathcal{D}$ is \emph{comonadic} if it has a right adjoint $G: \mathcal{D} \rightarrow \mathcal{C}$ and the comparison functor $J_{\mathbb{U}} : \mathcal{C} \rightarrow \mathcal{D}_{\mathbb{U}}$ is an equivalence of categories, where $\mathbb{U} = (U := F\circ G, \varepsilon, \Delta)$ is the resulting comonad on $\mathcal{D}$.
\end{defn}

The following result, sometimes known as \emph{Beck's Monadicity Theorem}, gives criterion for when a functor is comonadic. 

\begin{theorem}\emph{(The Dual Barr-Beck Theorem \cite[p.~212]{B}).}
\label{BarrBeck}
A functor $F: \mathcal{C} \rightarrow \mathcal{D}$ is comonadic if and only if
\begin{itemize}
\item[i)]$F$ has a right adjoint $G$;
\item[ii)]$F$ reflects isomorphisms; and
\item[iii)]if a pair $f,g : A \rightarrow B$ are morphisms in $\mathcal{C}$ such that $F(f), F(g)$ have a split equaliser $h: H \rightarrow F(A)$ in $\mathcal{D}$ then $f,g$ have an equaliser $e: E \rightarrow A$ in $\mathcal{C}$ such that $F(e)=h, F(E) = H$. 
\end{itemize}
\end{theorem}

\subsection{Crystals as coalgebras of a comonad}

\begin{defn}
Suppose we have a set $\mathbb{X}$ and, for each $x \in \mathbb{X}$, a pointed set $B(x)$. Let $\mathcal{C}_{\mathbb{X}}$ denote the category whose objects are sets $A$ equipped with a map $\pi_{A}:A \rightarrow \mathbb{X}$. Morphisms in this category are defined to be morphisms of pointed sets $A \sqcup \{0\} \xrightarrow{\psi} A' \sqcup \{0\}$ such that $\pi_{A}(a)=\pi_{A'}\phi(a)$ whenever $\phi(a) \neq 0$. We will denote by $F$ the functor
$$\mathcal{C}_{\mathbb{X}} \rightarrow \mathit{Set}_{\bullet}, \quad (A \xrightarrow{\pi_{A}} \mathbb{X}) \mapsto \bigsqcup\nolimits_{a \in A} B(\pi_{A}(a)).$$
For a morphism $A \sqcup \{0\} \xrightarrow{\psi} A' \sqcup \{0\}$ in $\mathcal{C}_{\mathbb{X}}$, $F(\phi)$ maps $B(\pi_{A}(a))$ isomorphically to $B(\pi_{A'}\phi(a))$ whenever $\phi(a) \neq 0$, and maps $B(\pi_{A}(a))$ to $0$ when $\phi(a) = 0$.
\end{defn}

\begin{lem}
\label{CrystalsInGeneralSetting}
If $\mathbb{X}=\Phi_{+}$ and $B(\alpha)$ are as previously defined for $\alpha \in \Phi_{+}$ then $\mathcal{C}_{\mathbb{X}} \cong \mathit{Crys}_{\mathfrak{g}}$. Furthermore, under this equivalence, $F$ is the forgetful functor to pointed sets.
\end{lem}

\begin{proof}
This equivalence is given by
$$\mathcal{C}_{\mathbb{X}} \rightarrow \mathit{Crys}_{\mathfrak{g}}, \quad (A\xrightarrow{\pi_{A}}\mathbb{X}) \mapsto \bigsqcup\nolimits_{a \in A} B(\pi_{A}(a)),$$
where a morphism $\psi:A \sqcup \{0\} \rightarrow A' \sqcup \{0\}$ in $\mathcal{C}_{\mathbb{X}}$ is mapped to the morphism of crystals $\bigsqcup\nolimits_{a \in A} B(\pi_{A}(a)) \rightarrow \bigsqcup\nolimits_{a' \in A'} B(\pi_{A'}(a'))$ where $B(\pi_{A}(a))$ is mapped isomorphically to $B(\pi_{A'}\phi(a))$ whenever $\phi(a) \neq 0$, and to $0$ when $\phi(a) = 0$. Its quasi-inverse is given by the functor
$$\mathit{Crys}_{\mathfrak{g}} \rightarrow \mathcal{C}_{\mathbb{X}}, \quad \bigsqcup\nolimits_{i \in I} B(\alpha_{i}) \mapsto (I, i \mapsto \alpha_{i}).$$
By Lemma \ref{SchursLemma}, a morphism of crystals $\bigsqcup\nolimits_{i \in I} B(\alpha_{i}) \rightarrow \bigsqcup\nolimits_{j \in J} B(\beta_{j})$ maps each $B(\alpha_{i})$ either isomorphically to some $B(\beta_{j})$, where $\alpha_{i}=\beta_{j} \in \mathbb{X}$, or to $0$. The resulting map $I \sqcup \{0\} \rightarrow J \sqcup \{0\}$ maps $i \mapsto j$ in the former case and $i \mapsto 0$ in the latter.
\end{proof}

\begin{defn}
Let $G$ denote the functor
$$G: \mathit{Set}_{\bullet} \rightarrow \mathcal{C}_{\mathbb{X}}, \quad X \mapsto G(X)$$
where $G(X)$ is the set $\bigsqcup_{x \in \mathbb{X}} \left( \text{Hom}(B(x),X) \setminus \{0\} \right)$
equipped with the map $\pi_{G(X)}$ which takes $f \in \text{Hom}(B(x),X) \setminus \{0\}$ to $x$. A map of pointed sets $\psi:X \rightarrow Y$ gives the map $G(\psi):G(X) \rightarrow G(Y)$ taking $f \in \text{Hom}(B(x),X)$ to $\phi \circ f \in \text{Hom}(B(x),Y)$.
\end{defn}

\begin{prop}
\label{FGAdjunction}
There is an adjunction $F \dashv G$ between $\mathcal{C}_{\mathbb{X}}$ and the category of pointed sets.
\end{prop}

\begin{proof}
Suppose we have a morphism $(A \xrightarrow{\pi_{A}} \mathbb{X}) \xrightarrow{f} (GX \xrightarrow{\pi_{GX}} \mathbb{X})$ in $\mathcal{C}_{\mathbb{X}}$. Each $a \in A$ is either mapped to $0$ or to an element of
$$\pi_{GX}^{-1}(\pi(a))=\text{Hom}(B(\pi(a)),X) \setminus \{0\}.$$
That is, $a$ is mapped to a function $f_{a} \in \text{Hom}(B(\pi(a)),X)$ which allows us to define a map $FA = \bigsqcup\nolimits_{a \in A} B(\pi_{A}(a)) \rightarrow X$. Conversely, a map of pointed sets $FA=\bigsqcup\nolimits_{a \in A} B(\pi_{A}(a)) \xrightarrow{g} X$ is given by a collection of maps
$$g_{a} \in \text{Hom}(B(\pi(a)),X) = (\text{Hom}(B(\pi(a)),X) \setminus \{0\}) \sqcup \{0\}.$$
Thus we get a map
$$A \sqcup \{0\} \rightarrow \bigsqcup_{a \in A}(\text{Hom}(B(\pi(a)),X) \setminus \{0\}) \sqcup \{0\} = GX \sqcup \{0\}$$
by taking $a \in A$ to the function $g_{a}$. These mutual inverses give the adjunction
$$\text{Hom}_{\mathit{Set}_{\bullet}}(FA, X) \cong \text{Hom}_{\mathcal{C}_{\mathbb{X}}}(A,GX).$$
Note that the unit of this adjunction is given by maps
$$FG(X) = \coprod\nolimits_{x \in \mathbb{X}}\coprod\nolimits_{\text{Hom}(B(x),X)} B(x) \rightarrow X,$$
for pointed sets $X$, where the copy of $B(x)$ indexed by $f \in \text{Hom}(B(x),X)$ is mapped to $X$ via $f$. The counit of the adjunction is given by maps
$$A \rightarrow GF(A)=\coprod\nolimits_{x \in \mathbb{X}} \text{Hom}(B(x),\coprod\nolimits_{a \in A}B(\pi_{A}(a))),$$
for $(A, \pi_{A})$ in $\mathcal{C}_{\mathbb{X}}$, where $a \in A$ is mapped to the inclusion of $B(\pi_{A}(a))$ into $\coprod_{a' \in A}B(\pi_{A}(a'))$.
\end{proof}

\begin{lem}
\label{FReflectsIsomorphisms}
The functor $F$ reflects isomorphisms.
\end{lem}

\begin{proof}
Let $\phi:A \sqcup \{0\} \rightarrow A' \sqcup \{0\}$ be a morphism in $\mathcal{C}_{\mathbb{X}}$, and suppose that $F(\phi):FA \rightarrow FA'$ is an isomorphism. This precisely means that $\phi(a) \neq 0$ for all $a \in A$ and each $B(\pi_{A'}(a'))$ for $a' \in A'$ is the image of some $B(\pi_{A}(a))$ for some $a \in A$. Thus $\phi$ is both monic and epic, hence an isomorphism. Its inverse it just its inverse as a map of pointed sets.
\end{proof}

\begin{lem}
\label{FPreservesEqualisers}
$\mathcal{C}_{\mathbb{X}}$ has, and $F$ preserves, all equalisers.
\end{lem}

\begin{proof}
Suppose we have parallel maps $f,g:A \sqcup \{0\} \rightarrow A' \sqcup \{0\}$ in $\mathcal{C}_{\mathbb{X}}$. Their equaliser is just $\{a \in A \mid f(a)=g(a)\}$. Likewise, the equaliser of $F(f)$ and $F(g)$ in $\mathit{Set}_{\bullet}$ is $\{b \in FA \mid F(f)(b)=F(g)(b)\}$. On each component $B(\pi_{A}(a))$ of $FA$, if $f(a)=g(a) \neq 0$ then they are both the same isomorphism
$$B(\pi_{A}(a)) \xrightarrow{\sim} B(\pi_{A'}(f(a)))=B(\pi_{A'}(g(a))).$$
If $f(a)=0=g(a)$ then $F(f)$ and $F(g)$ are both the zero map. Otherwise $F(f)$ and $F(g)$ must disagree on all non-zero elements of the component. Hence the equaliser of $F(f)$ and $F(g)$ is the union of $B(\pi_{A}(a))$ where $f(a)=g(a)$. This is $F(\{a \in A \mid f(a)=g(a)\})$, and the lemma is proved.
\end{proof}

\begin{theorem}
\label{GeneralClassificationAsCoalgebras}
The comonad $\mathbb{U}=(U=F \circ G, \eta, \mu)$ induces an equivalence of categories $J_{\mathbb{U}} : \mathcal{C}_{\mathbb{X}} \rightarrow \mathit{Set}_{\bullet \mathbb{U}}$ between $\mathcal{C}_{\mathbb{X}}$ and the category of algebras over the comonad $\mathbb{U}$ in $\mathit{Set}_{\bullet}$.
\end{theorem}

\begin{proof}
This follows from Theorem \ref{BarrBeck}, using Proposition \ref{FGAdjunction}, Lemma \ref{FReflectsIsomorphisms} and Lemma \ref{FPreservesEqualisers}.
\end{proof}

\begin{corollary}
\label{CrystalClassificationAsCoalgebras}
Setting $\mathbb{X}=\Phi_{+}$ and $B(\alpha)$ are as previously defined for $\alpha \in \Phi_{+}$, the comonad $\mathbb{U}=(U=F \circ G, \eta, \mu)$ gives an equivalence of categories $J_{\mathbb{U}} : \mathit{Crys}_{\mathfrak{g}} \rightarrow \mathit{Set}_{\bullet \mathbb{U}}$.
\end{corollary}

\begin{proof}
This follows from Theorem \ref{GeneralClassificationAsCoalgebras} and Lemma \ref{CrystalsInGeneralSetting}.
\end{proof}

\begin{remark}
Explicitly, from the proof of Proposition \ref{FGAdjunction}, we see that
$$U=FG: A \mapsto \bigsqcup_{\alpha \in \Phi_+} \bigsqcup_{\substack{f \in \text{Hom}(FB(\alpha),A) \\ f \neq 0}} F(B(\alpha)_{f})$$
with
$$\eta_{B(\alpha)} : B(\alpha) \rightarrow \bigsqcup_{\beta \in \Phi_+} \, \bigsqcup_{\substack{f \in \text{Hom}(FB(\beta),FB(\alpha)) \\ f \neq 0}} B(\beta)_{f},$$
$$b \mapsto (b)_{\text{id}_{FB(\alpha)}} \in B(\alpha)_{\text{id}_{FB(\alpha)}},$$
and
$$\varepsilon_{A} : \bigsqcup_{\alpha \in \Phi_+} \, \bigsqcup_{\substack{f \in \text{Hom}(FB(\alpha),A) \\ f \neq 0}} F(B(\alpha)_{f}) \rightarrow A,$$
$$(b)_{f} \mapsto f(b).$$
So
$$\Delta_{A} : \bigsqcup_{\alpha \in \Phi_+} \bigsqcup_{\substack{f \in \text{Hom}(FB(\alpha),A) \\ f \neq 0}} FB(\alpha)_{f} \rightarrow \bigsqcup_{\beta \in \Phi_+} \, \bigsqcup_{\substack{g \in \text{Hom}(FB(\beta),FG(A))  \\ g \neq 0}} FB(\beta)_{g}$$
maps the copy of $FB(\alpha)$ indexed by $f:FB(\alpha) \rightarrow A$ isomorphically to the copy of $FB(\alpha)$ indexed by $FB(\alpha) \cong FB(\alpha)_{f} \hookrightarrow FG(A)$. From here we can explicitly see the coalgebra structure of each $B(\alpha)$ over $FG$ is given by a map
$$\zeta: F(B(\alpha)) \rightarrow FG(F(B(\alpha)), \, \, \, b \mapsto (b)_{\text{id}_{F(B(\alpha))}}$$
which extends to the coalgebra structure of a general crystal $X= \bigsqcup _{j \in J} B(\beta_{j})$ as follows:
$$\zeta: F(X) \rightarrow FG(F(X)), \, \, \, b \mapsto (b)_{(F(B(\beta_{j})) \hookrightarrow FX)} \, \text{ for }b \in F(B(\beta_{j})).$$
\end{remark}

\begin{corollary}
\label{NoCrystalCoalgebra}
There is no coalgebra in $\mathit{Set}_{\bullet}$ whose category of comodules is equivalent to $\mathit{Crys}_{\mathfrak{g}}$ as categories over $\mathit{Set}_{\bullet}$.
\end{corollary}

\begin{proof}
Suppose there is a coalgebra $C$ in $\mathit{Set}_{\bullet}$ whose category of comodules is equivalent to $\mathit{Crys}_{\mathfrak{g}}$, and suppose this equivalence preserves the forgetful functor to $\mathit{Set}_{\bullet}$. Then the right adjoint to this forgetful functor, $G$, would be isomorphic to $C \otimes -:\mathit{Set}_{\bullet} \rightarrow C \text{-comod} \cong \mathit{Crys}_{\mathfrak{g}}$. Then $U \cong C \otimes -$, as a functor on $\mathit{Set}_{\bullet}$, preserves coproducts. However, by the explicit description of $U$, this is not the case and we reach a contradiction.
\end{proof}

\subsection{Recovering the crystal structure}

Given a pointed set $A$ with a coalgebra structure $(A, \zeta_{A})$ over our comonad $U=FG$, we know from the above that $A$ carries a crystal structure that has been forgotten by the forgetful functor $F$. In fact, there is a way of recovering this crystal structure from the coalgebra structure.

\begin{prop}
We regain the Kashiwara operator $\tilde{f}_i$ (and similarly $\tilde{e}_i$) on a $U$-coalgebra $A$ via the following composition:
$$A \xrightarrow[]{\zeta_{A}} FG(A) \xrightarrow[]{\tilde{f}_i} FG(A) \overset{\varepsilon_{A}}{\longrightarrow} A.$$
We also regain the weight function via
$$A \rightarrow FG(A) \rightarrow \Phi$$
where the last arrow is the map $(b)_{f} \mapsto \text{wt}(b)$.
\end{prop}

\begin{proof}
This follows from the explicit description of the $U$-coaction on a crystal as described in the previous remark.
\end{proof}

\subsection{The monoidal structure of $U$}

Recall that a bialgebra is simultaneously an algebra and a coalgebra where the structure maps are compatible. In the setting of functors, there is no analogous notion of a bimonad. The subtlety comes from the lack of symmetry when composing functors. There is no natural twist $A\circ B \Rightarrow B \circ A$ for functors $A,B$ on a category $\mathcal{C}$, and so, whilst the tensor product of two algebras in a symmetric monoidal category again gives an algebra, the composition of two monads does not naturally give a monad. So, if a functor $T$ on a category is both a monad and a comonad, we cannot simply ask that the comultiplication map $T \Rightarrow TT$ be a morphism of monads. Recall that, for a bialgebra $H$, the categories of comodules of $H$ inherit a monoidal structure. We wish to generalise this property of bialgebras that allows us to encode a monoidal structure on comodules. To generalise this, we recall the definition of a monidal functor. For more on these notions see \cite{Mo}, \cite{BV}, \cite{BLV} and \cite{PMC}.

\begin{defn}
Let $T:\mathcal{C} \rightarrow \mathcal{D}$ be a functor between monoidal categories. $T$ is said to be \emph{monoidal} if we equip $T$ with the data of a natural tranformation
$$\chi _{A,B}:T(A) \otimes T(B) \Rightarrow T(A \otimes B)$$
and a morphism $\mathbb{I} \rightarrow T(\mathbb{I})$, where $\mathbb{I}$ is taken to be the identity of the tensor product, satisfying the diagram
\begin{center}
\begin{tikzpicture}[node distance=6cm, auto]
  \node (A) {$T(A) \otimes (T(B) \otimes T(C))$};
  \node (B) [right=1cm of A] {$(T(A) \otimes T(B)) \otimes T(C)$};
  \node (C) [below=0.5cm of B] {$T(A \otimes B) \otimes T(C)$};
  \node (D) [below=0.5cm of A] {$T(A) \otimes T(B \otimes C)$};
  \node (E) [below=0.5cm of D] {$T(A \otimes (B \otimes C))$};
  \node (F) [below=0.5cm of C] {$T((A \otimes B) \otimes C)$};
  \draw[->] (A) to node {$\sim$} (B);
  \draw[->] (B) to node {$\chi_{A,B} \otimes \text{Id}$} (C);
  \draw[->] (C) to node {$\chi_{A \otimes B, C}$} (F);
  \draw[->] (A) to node {$\text{Id} \otimes \chi_{B,C}$} (D);
  \draw[->] (D) to node {$\chi_{A, B \otimes C}$} (E);
  \draw[->] (E) to node {$\sim$} (F);
\end{tikzpicture},
\end{center}
and such that the compositions
$$T(A) \cong T(A) \otimes \mathbb{I} \rightarrow T(A) \otimes T(\mathbb{I}) \rightarrow T(A \otimes \mathbb{I}) \cong T(A),$$
$$T(A) \cong \mathbb{I} \otimes T(A) \rightarrow T(\mathbb{I}) \otimes T(A) \rightarrow T(\mathbb{I} \otimes A) \cong T(A),$$
are the identity on $T(A)$ for all $A$, $B$ and $C$ in $\mathcal{C}$. We say that $T$ is \emph{strong monoidal} if $\chi_{A,B}$ and $\mathbb{I} \rightarrow T(\mathbb{I})$ are isomorphisms. If $T$ is a monoidal comonad on a monoidal category, we will call it is a \emph{bicomonad} if the diagram
\begin{center}
\begin{tikzpicture}[node distance=6cm, auto]
  \node (A) {$TT(A \otimes B)$};
  \node (B) [left=1.3cm of A] {$T(T(A) \otimes T(B))$};
  \node (C) [left=1.3cm of B] {$TT(A) \otimes TT(B)$};
  \node (D) [above=1cm of A] {$T(A \otimes B)$};
  \node (E) [above=1cm of C] {$T(A) \otimes T(B)$};
  \draw[<-] (A) to node {$T(\chi_{A,B})$} (B);
  \draw[<-] (B) to node {$\chi _{T(A),T(B)}$} (C);
  \draw[<-] (A) to node {$\Delta_{A \otimes B}$} (D);
  \draw[<-] [swap] (D) to node {$\chi_{A,B}$} (E);
  \draw[<-] (C) to node {$\Delta_{A} \otimes \Delta_{B}$} (E);
\end{tikzpicture},
\end{center}
commutes and $\chi_{A,B} \circ \epsilon_{A \otimes B} = \epsilon_{A} \otimes \epsilon_{B}$ as maps $T(A) \otimes T(B) \rightarrow A \otimes B$ for all $A$ and $B$ in $\mathcal{C}$.
\end{defn}

\begin{remark}
For a comonad $\mathbb{U}$, Proposition 1.4 of \cite{Mo} shows that the property of being a bicomonad gives a monoidal structure on the category of coalgebras. The coaction on a tensor product of two coalgebras is given by the composition
$$A \otimes B \rightarrow T(A) \otimes T(B) \rightarrow T(A \otimes B)$$
where the first arrow is given by the respective coactions of $A$ and $B$, and the second given by $\chi$. In fact, Moerdijk proves the following in \cite{Mo}.
\end{remark}

\begin{theorem}[\cite{Mo}]
\label{Moerdijk'sTheorem}
Let $\mathbb{U}= (U, \Delta, \epsilon)$ be a comonad on a monoidal category $\mathcal{C}$. Then monoidal structures on $\mathcal{C}_{\mathbb{U}}$ such that the forgetful functor $F$ to $\mathcal{C}$ is strong monoidal correspond to bicomonad structures on $\mathbb{U}$.
\end{theorem}

\begin{proof}
This is Theorem 7.1 of \cite{Mo}. Suppose we have endowed $\mathcal{C}_{\mathbb{U}}$ with a monoidal structure $(\otimes, \mathbb{I})$ such that $F$ is strong monoidal. Let $G:\mathcal{C} \rightarrow \mathcal{C}_{\mathbb{U}}$, $C \mapsto (UC,\Delta_{C})$ denote the right adjoint to $F$ mapping an object of $\mathcal{C}$ to its free coalgebra. Then $U=FG$ and we obtain $\chi_{A,B}:U(A) \otimes U(B) \rightarrow U(A \otimes B)$ as the image of $\epsilon_{A} \otimes \epsilon_{B}$ under the composition
$$\begin{array}{rcl}
\text{Hom}(FGA \otimes FGB,A \otimes B) &\cong& \text{Hom}(F(GA \otimes GB), A \otimes B)\\
&\cong& \text{Hom}(GA \otimes GB, G(A \otimes B))\\
&\rightarrow& \text{Hom}(F(GA \otimes GB), FG(A \otimes B))\\
&\cong& \text{Hom}(FGA \otimes FGB, FG(A \otimes B)).
\end{array}$$
The morphism $U(\mathbb{I}) \rightarrow \mathbb{I}$ is given by the counit $\epsilon_{\mathbb{I}}$.
\end{proof}

\begin{prop}
The monoidal structure on $\mathbb{U}$ corresponding to the monoidal structure on $\mathit{Crys}_{\mathfrak{g}}$ under the equivalence in Corollary \ref{CrystalClassificationAsCoalgebras} is given as follows.
For each $b \otimes b' \in FB(\alpha)_{f} \otimes FB(\beta)_{g} \subset U(A) \otimes U(B)$ indexed by $f:FB(\alpha) \rightarrow A$ and $g:FB(\beta) \rightarrow B$ there is some $\gamma_{b,b'} \in \Gamma _{\alpha, \beta}$ such that the image $b \cdot b'$ of $b \otimes b'$ in the decomposition $B(\alpha) \otimes B(\beta) \cong \bigsqcup_{\gamma \in \Gamma _{\alpha, \beta}} B(\gamma)$ lies in the component $B(\gamma_{b,b'})$. Then we define
$$\chi_{A,B}:U(A) \otimes U(B) \rightarrow U(A \otimes B)$$
by mapping $b \otimes b'$ to $b \cdot b'$ in the copy of $B(\gamma_{b,b'})$ indexed by the map
$$B(\gamma_{b,b'}) \hookrightarrow \bigsqcup_{\gamma \in \Gamma _{\alpha, \beta}} B(\gamma) \cong B(\alpha) \otimes B(\beta) \xrightarrow{f \otimes g} A \otimes B.$$
We define a map $\mathbb{I} \rightarrow U(\mathbb{I})$, where $\mathbb{I}=\{0,1\}$ is the monoidal unit in $\mathit{Set}_{\bullet}$, by mapping $1$ to $b_{0} \in B(0)$ indexed by the map $FB(0) \xrightarrow{\sim} \mathbb{I}$, $b_{0} \mapsto 1$.
\end{prop}

\begin{proof}
This result follows from the proof of Theorem \ref{Moerdijk'sTheorem}. The image of the map $\varepsilon_{A} \otimes \varepsilon_{B}$ under the isomorphism
$$\begin{array}{rcl}
\text{Hom}(FGA \otimes FGB,A \otimes B) &\cong& \text{Hom}(F(GA \otimes GB), A \otimes B)\\
&\cong& \text{Hom}(GA \otimes GB, G(A \otimes B))
\end{array}$$
is the map taking $b \otimes b' \in B(\alpha)_{f} \otimes B(\beta)_{g} \subset G(A) \otimes G(B)$ indexed by $f:FB(\alpha) \rightarrow A$ and $g:FB(\beta) \rightarrow B$ to $b \cdot b'$ in the copy of the irreducible crystal $B(\gamma_{b,b'}) \subset \bigsqcup_{\gamma \in \Gamma _{\alpha, \beta}} B(\gamma) \cong B(\alpha) \otimes B(\beta)$ in which it lies indexed by the restriction of $f \otimes g$ to this component.
\end{proof}

\section{A crystal bialgebra}
\label{WorkingOverZ}

\subsection{The bialgebra $\mathbb{B}$}

\begin{defn}
For $\alpha \in \Phi_{+}$ let $\mathbb{B}(\alpha)$ be the free abelian group $\mathbb{Z}B(\alpha)$. Let $\iota_{\alpha}$ and $\epsilon_{\alpha}$ denote the homomorphisms
$$\begin{array}{rcll}
\iota_{\alpha}&:&\mathbb{Z} \rightarrow \mathbb{B}(-\alpha) \otimes \mathbb{B}(\alpha),& 1 \mapsto \sum_{b \in B(\alpha)} b^{\vee} \otimes b,\\
\varepsilon_{\alpha}&:& \mathbb{B}(\alpha) \otimes \mathbb{B}(-\alpha) \rightarrow \mathbb{Z},& b \otimes b' \mapsto \delta_{b,b'^{\vee}},
\end{array}$$
called the \emph{coevaluation} and \emph{evaluation} respectively.
\end{defn}

\begin{prop}
The composition
$$\mathbb{B}(\alpha) \cong \mathbb{B}(\alpha)\otimes \mathbb{Z} \overset{\text{id} \otimes \iota_{\alpha}}{\longrightarrow} \mathbb{B}(\alpha) \otimes \mathbb{B}(-\alpha) \otimes \mathbb{B}(\alpha) \overset{\varepsilon_{\alpha} \otimes \text{id}}{\longrightarrow} \mathbb{B}(\alpha)$$
agrees with the identity. Hence $\mathbb{B}(-\alpha)$ is dual to $\mathbb{B}(\alpha)$ in the category of free abelian groups.
\end{prop}

\begin{proof}
This follows since the image of $b \otimes b' \in B(\alpha) \otimes B(-\alpha)$ under this composition is $\sum_{d \in B(\alpha)} \delta_{b,d} d \otimes b' = b \otimes b'$.
\end{proof}

\begin{defn}
Let $\mathbb{B}$ denote the free abelian group on the crystal $\mathcal{B}$,
$$\mathbb{B} := \mathbb{Z}\mathcal{B} = \oplus_{\alpha \in \Phi_{+}} \mathbb{B}(\alpha) \otimes \mathbb{B}(-\alpha).$$
Let $\Delta : \mathbb{B} \rightarrow \mathbb{B} \otimes \mathbb{B}$ denote the homomorphism defined on each summand $\mathbb{B}(\alpha) \otimes \mathbb{B}(-\alpha)\cong \mathbb{B}(\alpha) \otimes \mathbb{Z} \otimes \mathbb{B}(-\alpha)$ by $\text{id} \otimes \iota_{\alpha} \otimes \text{id}$, let $\varepsilon:\mathbb{B} \rightarrow \mathbb{Z}$ be the sum of the maps $\varepsilon_{\alpha}$. Let $\mu: \mathbb{B} \otimes \mathbb{B} \rightarrow \mathbb{B}$ and $\eta:\mathbb{Z} \cong \mathbb{B}(0) \otimes \mathbb{B}(-0) \hookrightarrow \mathbb{B}$ be the homomorphisms induced by the multiplication and unit in Definition \ref{Multiplication on crystal B}.
\end{defn}

\begin{prop} The maps $\eta$, $\mu$, $\varepsilon$ and $\Delta$ make $\mathbb{B}$ a $\mathbb{Z}$-bialgebra.
\end{prop}

\begin{proof}
The fact that $(\mathbb{B},\mu,\eta)$ forms a $\mathbb{Z}$-algebra follows from Proposition \ref{Crystal Algebra B}. Both $(\Delta \otimes \text{Id}) \circ \Delta$ and $(\text{Id} \otimes \Delta) \circ \Delta$ can be identified with the map $\text{Id} \otimes \iota_{\alpha} \otimes \iota_{\alpha} \otimes \text{Id}$ on
$$\mathbb{B}(\alpha) \otimes \mathbb{B}(-\alpha) \cong \mathbb{B}(\alpha) \otimes \mathbb{Z} \otimes \mathbb{Z} \otimes \mathbb{B}(-\alpha),$$
so the comultiplication is coassociative. Furthermore,
$$\begin{array}{rcl}
\sum_{d \in B(\alpha)} \varepsilon(b \otimes d^{\vee}) \ d \otimes b' &=& \sum_{d \in B(\alpha)} \delta_{b,d} \ d \otimes b'\\
&=& b \otimes b'\\
&=& \sum_{d \in B(\alpha)} \delta_{d,b'^{\vee}} \ b \otimes d^{\vee}\\
&=& \sum_{d \in B(\alpha)} \varepsilon(d \otimes b') \ b \otimes d^{\vee},
\end{array}$$
so $\varepsilon$ acts as a counit. It remains to verify that $\Delta$ and $\varepsilon$ are $\mathbb{Z}$-algebra homomorphisms. Let $b,b' \in B(\alpha)$ and $d,d' \in B(\beta)$. If $(b \cdot d)$ and $(d' \cdot b')^{\vee}$ lie in different irreducible crystals then $\Delta \circ \mu ((b \otimes b') \otimes (d \otimes d'))=0$. Also,
$$(b \otimes b') \otimes (d \otimes d')\xmapsto{\mu_{\mathbb{B} \otimes \mathbb{B}} \circ \Delta_{\mathbb{B} \otimes \mathbb{B}}}\sum_{\substack{b'' \in B(\alpha) \\ d'' \in B(\beta)}} \mu(b \otimes b''^{\vee} \otimes  d \otimes d''^{\vee}) \otimes \mu(b'' \otimes b' \otimes d'' \otimes d'),$$
the nonzero terms of which only occur when both $b\otimes d$ and $(d''^{\vee} \otimes b''^{\vee})^{\vee}=b'' \otimes d''$ lie in the same component, and $b'' \otimes d''$ and $(d' \otimes b')^{\vee}$ lie in the same component. Since this never occurs, this sum must also be zero. Now suppose that $(b \otimes d)$ and $(d' \otimes b')^{\vee}$ do lie in the same irreducible component, $B(\gamma)$ say. In this case we have
$$(b \otimes b') \otimes (d \otimes d') \overset{\mu}{\mapsto} (b \cdot d) \otimes (d' \cdot b') \overset{\Delta}{\mapsto} \sum_{c \in B(\gamma)} ((b \cdot d) \otimes c^{\vee}) \otimes (c \otimes (d' \cdot b'))$$
whilst
$$\begin{array}{rcl}
(b \otimes b') \otimes (d \otimes d') &\xmapsto{\Delta \otimes \Delta}& \sum_{\substack{b'' \in B(\alpha)\\ d'' \in B(\beta)}} b \otimes b''^{\vee} \otimes b'' \otimes b' \otimes d \otimes d''^{\vee} \otimes d'' \otimes d'\\
&\xmapsto{\mu_{\mathbb{B} \otimes \mathbb{B}}}& \sum_{\substack{b'' \in B(\alpha)\\ d'' \in B(\beta)}} \mu(b \otimes b''^{\vee} \otimes  d \otimes d''^{\vee}) \otimes \mu(b'' \otimes b' \otimes d'' \otimes d')\\
&=& \sum_{\substack{b'' \in B(\alpha) \\ d'' \in B(\beta) \\ b'' \cdot d'' \in B(\gamma)}}(b \cdot d) \otimes (b'' \cdot d'')^{\vee} \otimes (b'' \cdot d'') \otimes (d' \cdot b')\\
&=& \sum_{c \in B(\gamma)} (b \cdot d) \otimes c^{\vee} \otimes c \otimes (d' \cdot b').
\end{array}$$
So $\Delta$ is an algebra homomorphism. Similarly, if we say $b \otimes d$ and $b'^{\vee} \otimes d'^{\vee}$ lie in the same component,
$$\begin{array}{rcccl}
\epsilon((b \otimes b') \cdot (d \otimes d')) &=& \epsilon(b \cdot d \otimes d' \cdot b') &=& \delta_{(b \cdot d)^{\vee}, d' \cdot b'}\\
&=& \delta_{d^{\vee} \cdot b^{\vee}, d' \cdot b'}&=& \delta_{d^{\vee},d'} \delta_{b^{\vee},b'}\\
&=& \epsilon(b \otimes b') \epsilon (d \otimes d').
\end{array}$$
since $d^{\vee} \cdot b^{\vee} = d' \cdot b'$ if and only if $d^{\vee} = d'$ and $b^{\vee}=b'$. The case when they do not lie in the same component is trivial, hence $\epsilon$ is an algebra homomorphism too. Thus we have our result.
\end{proof}

\begin{defn}
Let $\mathbb{B}_{\lambda} = \text{Span}_{\mathbb{Z}} \{ b \otimes b' \in \mathcal{B} \mid \text{wt}(b)+\text{wt}(b')= \lambda \}$ for $\lambda \in \Phi$.
\end{defn}

\begin{prop}
We have $\mathbb{B} = \bigoplus_{\lambda \in \Phi} \mathbb{B}_{\lambda}$ with $\mathbb{B}_{\lambda} \cdot \mathbb{B}_{\lambda'} \subset \mathbb{B}_{\lambda + \lambda'}$ and $\Delta(\mathbb{B}_{\lambda}) \subset \bigoplus_{\lambda = \lambda' + \lambda''} \mathbb{B}_{\lambda'} \otimes \mathbb{B}_{\lambda''}$, so $\mathbb{B}$ is a graded bialgebra.
\end{prop}

\begin{proof}
Let $b \otimes b' \in B(\alpha) \otimes B(-\alpha)$ and $d \otimes d' \in B(\beta) \otimes B(-\beta)$. Then their product is either $0$ or $(b \cdot d) \otimes (d' \cdot b')$, and
$$\text{wt}(b \cdot d) + \text{wt}(b' \cdot d')=\text{wt}(b)+\text{wt}(d) + \text{wt}(b')+\text{wt}(d').$$
So $\mathbb{B}_{\lambda} \cdot \mathbb{B}_{\lambda'} \subset \mathbb{B}_{\lambda + \lambda'}$. Also, $\Delta(b \otimes b')=\sum_{d \in B(\alpha)} b \otimes d \otimes d^{\vee} \otimes b'$ and
$$
\begin{array}{rcl}
\text{wt}(b)+\text{wt}(d)+\text{wt}(d^{\vee})+\text{wt}(b')&=&\text{wt}(b)+\text{wt}(d)-\text{wt}(d)+\text{wt}(b')\\
&=&\text{wt}(b)+\text{wt}(b').
\end{array}$$
So $\Delta(\mathbb{B}_{\lambda}) \subset \bigoplus_{\lambda = \lambda' + \lambda''} \mathbb{B}_{\lambda'} \otimes \mathbb{B}_{\lambda''}$.
\end{proof}

\begin{prop}
If we take a basis of $\Phi$ of fundamental weights $\{\Lambda_{i} \mid i \in I\}$ then $\mathbb{B}$ is generated as an algebra by the $B(\Lambda_{i}) \otimes B(-\Lambda_{i})$ for $i \in I$.
\end{prop}
\begin{proof}
For each $\sum_{i} n_{i} \Lambda_{i} \in \Phi_+$ the surjection
$$B(\Lambda_{1})^{\otimes n_{1}} \otimes ... \otimes B(\Lambda_{k})^{\otimes n_{k}} \rightarrow B(\sum_{i} n_{i} \Lambda_{i})$$
gives a surjection
$$\bigotimes_{i=1}^{k}(B(\Lambda_{i}) \otimes B(-\Lambda_{i}))^{\otimes n_{i}} \rightarrow B(\sum_{i} n_{i} \Lambda_{i}) \otimes B(-\sum_{i} n_{i} \Lambda_{i})$$
onto a basis of $\mathbb{B}(\sum_{i} n_{i} \Lambda_{i}) \otimes \mathbb{B}(-\sum_{i} n_{i} \Lambda_{i})$.
\end{proof}

\begin{prop}
\label{BPresentationForsl2}
Suppose $\mathfrak{g}=\mathfrak{sl}_{2}$. Then $\mathbb{B}$ is the quotient of the free algebra $\mathbb{Z}\langle a,b,c,d \rangle$ by the relations
$$cb=bc=db=dc=ba=ca = 0, \, \, \, da=1$$
with comultiplication
$$\begin{array}{rclrcl}
\Delta(a)&=&a \otimes a + b \otimes c,& \Delta(b)&=&a \otimes b + b \otimes d,\\
\Delta(c)&=& c \otimes a + d \otimes c,& \Delta(d)&=&c \otimes b + d \otimes d.
\end{array}$$
\end{prop}

\begin{proof}
In the case of $\mathfrak{sl}_{2}$, the fundamental weight is $1 \in \mathbb{N}$, and $B(1)$ has crystal graph $u_{0}^{(1)} \rightarrow u_{1}^{(1)}$. So we have four generators in $B(1) \otimes B(-1)$, namely
$$\begin{array}{rclrcl}
a&=&u_{1}^{(1)} \otimes (u_{1}^{(1)})^{\vee},&b&=&u_{0}^{(1)} \otimes (u_{1}^{(1)})^{\vee},\\
c&=&u_{1}^{(1)} \otimes (u_{0}^{(1)})^{\vee},&d&=&u_{0}^{(1)} \otimes (u_{0}^{(1)})^{\vee}.
\end{array}$$
These generators have the given comultiplication. It follows from Definiton \ref{CrystalTensor} and the diagram in Example \ref{CrystalTensorsl2} that
$$u_{p}^{(n)} \cdot u_{q}^{(m)} =
\begin{cases}
u_{p}^{(m+n-2q)} \in B(m+n-2q) & \text{ if } p+q \leq n, \\
u_{2p+q-n}^{(m-n-2p)} \in B(m-n-2p) & \text{ if } p+q < n.
\end{cases}$$
Then $B(1)^{\otimes n} \twoheadrightarrow B(n)$ maps $(u_{1}^{(1)})^{\otimes k} \otimes (u_{0}^{(1)})^{\otimes n-k}$ to $u_{k}^{(n)}$. From this it follows that
$$\begin{array}{rcl}
u_{k}^{(n)} \otimes (u_{l}^{(n)})^{\vee} &=&
\begin{cases} 
 a^{l}c^{k-l}d^{n-k} & \text{if } k \geq l, \\
 a^{k}b^{l-k}d^{n-l} & \text{if } k \leq l.
\end{cases}
\end{array}$$
Furthermore, the multiplication in $\mathcal{B}$ can be computed as
$$(u_{p}^{(n)} \otimes (u_{q}^{(n)})^{\vee}) \cdot (u_{r}^{(m)} \otimes (u_{s}^{(m)})^{\vee}) =
\begin{cases}
u_{p}^{(m+n-2r)} \otimes (u_{q}^{(m+n-2r)})^{\vee} & \substack{\text{ if } p+r \leq n,\\ q+s \neq n,\\ r=s,} \\
u_{p}^{(m+n-2r)} \otimes (u_{2q+s-n}^{(m+n-2r)})^{\vee} & \substack{\text{ if } p+r \leq n,\\ q+s > n,\\ r=n+q,} \\
u_{2p+r-n}^{(m-n-2p)} \otimes (u_{2q+s-n}^{(m-n-2p)})^{\vee} & \substack{\text{ if } p+r > n,\\ q+s > n,\\ p=q,} \\
u_{2p+r-n}^{(m-n-2p)} \otimes (u_{q}^{(m-n-2p)})^{\vee} & \substack{\text{ if } p+r > n,\\ q+s \neq n,\\ n+p=s,} \\
0 & \text{otherwise}.
\end{cases}$$
Rewriting this in terms of the generators $a,b,c$ and $d$ this becomes
$$
\begin{array}{rcl}
(a^{i}c^{j}d^{k})\cdot(a^{r}c^{s}d^{t})&=&
\begin{cases}
a^{i}c^{j}d^{t+k-r} & \text{ if } s=0, r \leq k\\
a^{i+r-k}c^{s}d^{t} & \text{ if } j=0, r \geq k\\
a^{i}c^{j+s}d^{t} & \text{ if } r = k\\
0 & \text{ otherwise,}
\end{cases}\\

(a^{i}c^{j}d^{k})\cdot(a^{r}b^{s}d^{t})&=&
\begin{cases}
a^{i}c^{j}d^{t+k-r} & \text{ if } s=0, r \leq k\\
a^{i+r-k}b^{s}d^{t} & \text{ if } j=0, r \geq k\\
0 & \text{ otherwise,}
\end{cases}\\

(a^{i}b^{j}d^{k})\cdot(a^{r}c^{s}d^{t})&=&
\begin{cases}
a^{i}b^{j}d^{t+k-r} & \text{ if } s=0, r \leq k\\
a^{i+r-k}c^{s}d^{t} & \text{ if } j=0, r \geq k\\
0 & \text{ otherwise,}
\end{cases}\\

(a^{i}b^{j}d^{k})\cdot(a^{r}b^{s}d^{t})&=&
\begin{cases}
a^{i}b^{j}d^{t+k-r} & \text{ if } s=0, r \leq k\\
a^{i+r-k}b^{s}d^{t} & \text{ if } j=0, r \geq k\\
a^{i}b^{j+s}d^{t} & \text{ if } r = k\\
0 & \text{ otherwise.}
\end{cases}
\end{array}$$
This shows that the multiplication is completely determined by the given relations.
\end{proof}

\begin{remark}
The presentation above for $\mathbb{B}$ is closely related to a presentation of the quantum coordinate ring, which we discuss at the end of the paper.
\end{remark}

\subsection{The comodules of $\mathbb{B}$}

\begin{defn}
\label{Functor Crys to B-Comod}
For each $\alpha \in \Phi_+$ we can give $\mathbb{B}(\alpha)$ a $\mathbb{B}$-comodule structure via the following map:
$$\mathbb{B}(\alpha) \cong \mathbb{B}(\alpha) \otimes \mathbb{Z} \rightarrow \mathbb{B}(\alpha) \otimes \mathbb{B}(-\alpha) \otimes \mathbb{B}(\alpha) \hookrightarrow \mathbb{B} \otimes \mathbb{B}(\alpha)$$
$$b \mapsto \sum_{b' \in B(\alpha)} b \otimes b'^{\vee} \otimes b'.$$
This induces a functor $\mathit{Crys}_{\mathfrak{g}} \rightarrow \mathbb{B}\text{-comod}$.
\end{defn}

\begin{defn}
\label{DefinitionAbb'alpha}
Let $M$ be a $\mathbb{B}$-comodule. For each $\alpha \in \Phi_{+}$ and each $b,b' \in B(\alpha)$ let us denote by $A_{b,b'}^{\alpha}$ the $\mathbb{Z}$-linear endomorphism of $M$ defined uniquely by the property that
$$\Delta_{M}(m)=\sum\nolimits_{\substack{\alpha \in \Phi_{+} \\ b,b' \in B(\alpha)}} b \otimes b'^{\vee} \otimes A_{b,b'}^{\alpha}(m) \quad \text{ for all }m \in M.$$
Let us denote by $M_{b}^{\alpha}$ the image of $M$ under $A_{b,b}^{\alpha}$ for $b \in B(\alpha)$, $\alpha \in \Phi_{+}$, and let $M^{\alpha}=\sum_{b \in B(\alpha)}M_{b}^{\alpha}$.
\end{defn}

\begin{lem}
With notation as in Definition \ref{DefinitionAbb'alpha},
$$A^{\beta}_{d,d'} A^{\alpha}_{b,b'} = \delta_{\alpha, \beta} \delta_{b',d} A^{\alpha}_{b,d'}, \, \, \sum_{\alpha \in \Phi_+} \sum_{b \in B(\alpha)} A^{\alpha}_{b,b} = \text{Id}_{M},$$
as automorphisms of $M$ for all $\alpha, \beta \in \Phi_{+}$, $b,b' \in B(\alpha)$, $d,d' \in B(\beta)$. This latter relation makes sense since, for each $m \in M$, $A_{b,b'}^{\alpha}(m)=0$ for all but finitely many $b,b',\alpha$. Hence
$$M=\bigoplus\nolimits_{\alpha \in \Phi_{+}} M^{\alpha}=\bigoplus\nolimits_{\substack{\alpha \in \Phi_{+} \\ b \in B(\alpha)}} M_{b}^{\alpha}$$
and $A_{b,b'}^{\alpha}$ restrict to isomorphisms $M_{b}^{\alpha} \rightarrow M_{b'}^{\alpha}$.
\end{lem}

\begin{proof}
Since $M$ is a comodule, we have
$$\sum_{\alpha \in \Phi_+} \sum_{b,b' \in B(\alpha)} \sum_{\beta \in \Phi} \sum_{d,d' \in B(\beta)} b \otimes b'^{\vee} \otimes d \otimes d'^{\vee} \otimes A^{\beta}_{d,d'} A^{\alpha}_{b,b'}(m)$$
$$= \sum_{\alpha \in \Phi_+} \sum_{b,b' \in B(\alpha)} \sum_{d \in B(\alpha)} b \otimes d \otimes d^{\vee} \otimes b'^{\vee} \otimes A^{\alpha}_{b,b'}(m)$$
and
$$m = \sum_{\alpha \in \Phi_+} \sum_{b \in B(\alpha)}  A^{\alpha}_{b,b}(m)$$
for all $m \in M$, from which the relations follow. These imply that $A_{b,b}^{\alpha}$ form a set of perpendicular idempotents, which give the direct sum decomposition. Also, $A_{b',b}^{\alpha}A_{b,b'}^{\alpha}=A_{b,b}^{\alpha}$, which is the identity on $M_{b}^{\alpha}$, and if $m=A_{b,b}^{\alpha}(m) \in M_{b}^{\alpha}$ then $A_{b,b'}(m)=A_{b',b'}A_{b,b'}(m) \in M_{b'}^{\alpha}$.
\end{proof}

\begin{defn}
Let us denote by $\mathbb{B}\text{-comod}^{\text{free}}$ the category whose objects are $\mathbb{B}$ comodules $M$ such that each $M_{b}^{\alpha}$ is a free $\mathbb{Z}$-modules.
\end{defn}

\begin{lem}
For $M$ in $\mathbb{B}\text{-comod}^{\text{free}}$ and $\alpha \in \Phi_{+}$ we may endow the pointed set
$$\mathcal{C}(M^{\alpha}):=\left(\bigsqcup_{b \in B(\alpha)} M_{b}^{\alpha} \setminus \{0\}\right) \sqcup \{0\}$$
with the structure of a crystal by setting $\tilde{f}_{i}(m)$ for $m \in M_{b}^{\alpha}$ to be $\tilde{f}_{i}(m)=A^{\alpha}_{\tilde{f}b,b}(m)$ if $\tilde{f}_{i}b \neq 0$ and $\tilde{f}_{i}(m)=0$ otherwise, and likewise for $\tilde{e}_{i}$. Hence we may endow the pointed set $\mathcal{C}(M):= \bigsqcup_{\alpha \in \Phi_{+}}\mathcal{C}(M^{\alpha})$ with a crystal structure. Furthermore, $\mathcal{C}$ gives a functor $\mathbb{B}\text{-comod}^{\text{free}} \rightarrow \mathit{Crys}_{\mathfrak{g}}$ by restricting a morphism $M \rightarrow M'$ in $\mathbb{B}$-comod to the components
$$\bigsqcup_{b \in B(\alpha)} M_{b}^{\alpha} \setminus \{0\} \rightarrow \bigsqcup_{b \in B(\alpha)} M_{b}^{\alpha} \setminus \{0\}.$$
\end{lem}

\begin{proof}
The fact that we indeed have a crystal structure on $\mathcal{C}(M^{\alpha})$ follows from the observation that we may identify $\bigsqcup_{M^{\alpha}_{b_{\alpha}} \setminus \{0\}} B(\alpha) \cong \mathcal{C}(M^{\alpha})$ under the mapping that takes $b$ in the copy of $B(\alpha)$ indexed by a nonzero $m \in M_{b_{\alpha}}^{\alpha}
$ to the nonzero $A_{b_{\alpha},b}^{\alpha}(m) \in M_{b}^{\alpha}$. Since comodule homomorphisms commute with the operators $A_{b,b'}^{\alpha}$ it follows that the restriction of such a homomorphism gives a morphism of crystals.
\end{proof}

\begin{theorem}
\label{B-Comodules are free on crystals}
The functor $\mathit{Crys}_{\mathfrak{g}} \rightarrow \mathbb{B}\text{-comod}^{\text{free}}$ described in Definition \ref{Functor Crys to B-Comod} is essentially surjective. Furthermore it has a right adjoint
$$\mathcal{C}:\mathbb{B}\text{-comod}^{\text{free}} \rightarrow \mathit{Crys}_{\mathfrak{g}}.$$
\end{theorem}

\begin{proof}
We first prove essential surjectivity. Let $M$ be in $\mathbb{B}\text{-comod}^{\text{free}}$, and let $X^{\alpha}_{b_{\alpha}}$ be a free basis of $M_{b_{\alpha}}^{\alpha}$ for each $\alpha \in \Phi_{+}$. Given $\alpha \in \Phi_{+}$ and $b \in B(\alpha)$ let $X_{b}^{\alpha}=A_{b_{\alpha},b}^{\alpha}X_{b_{\alpha}}^{\alpha}$, which is a free basis of $M_{b}^{\alpha}$. We may endow
$$X:=\left(\bigsqcup\nolimits_{\alpha \in \Phi_{+}} \bigsqcup\nolimits_{b \in B(\alpha)} X_{b}^{\alpha}\right) \sqcup \{0\}$$
with a crystal structure by viewing it as a subset of $\mathcal{C}(M)$ closed under the action of the Kashiwara operators $\tilde{f}_{i}$ and $\tilde{e}_{i}$. Under the identification of $\mathcal{C}(M)$ with $\bigsqcup_{\alpha \in \phi_{+}}\bigsqcup_{M^{\alpha}_{b_{\alpha}} \setminus \{0\}} B(\alpha)$, $X$ corresponds to the disjoint union of copies of $B(\alpha)$ indexed over elements of $X^{\alpha}_{b_{\alpha}}$. Its image in $\mathbb{B}\text{-comod}^{\text{free}}$ under the functor in Definition \ref{Functor Crys to B-Comod} is the free abelian group $M$ with the $\mathbb{B}$-coaction
$$x \mapsto \sum_{b \in B(\alpha)} b_{\alpha} \otimes b^{\vee} \otimes A_{b_{\alpha},b}(x)$$
for $x \in X_{b_{\alpha}}^{\alpha}$. This is just the usual coaction on $M_{b_{\alpha}}^{\alpha} \subset M$, which generates $M$ as a comodule, hence these two coactions must agree on $M$. So $M$ is the image of $X$ under the functor in Definition \ref{Functor Crys to B-Comod}.
\\

We now prove the adjuction. It is enough to exhibit a natural isomorphism
$$\text{Hom}(\mathbb{B}(\alpha),M) \cong \text{Hom}(B(\alpha),\mathcal{C}(M))$$
for each $\alpha \in \Phi_{+}$. First, note that a morphism $f$ of $\mathbb{B}$ comodules $f:\mathbb{B}(\alpha) \rightarrow M$ commutes with the operators $A_{b,b'}^{\alpha}$. Thus $f$ maps $\mathbb{B}(\alpha)_{b}^{\alpha}$ to $M_{b}^{\alpha}$. In fact, $f$ is entirely determined by the restriction
$$f_{b_{\alpha}}:\mathbb{Z} \cong \mathbb{B}(\alpha)_{b_{\alpha}}^{\alpha} \rightarrow M_{b_{\alpha}}^{\alpha}$$
since, on $\mathbb{B}(\alpha)_{b}^{\alpha}$, $f$ is given by $A_{b_{\alpha},b}^{\alpha} f_{b_{\alpha}} A_{b,b_{\alpha}}^{\alpha}$. Thus $f$ amounts to a choice of element in $M_{b_{\alpha}}^{\alpha}$. Likewise, since $\mathcal{C}(M) \cong \bigsqcup_{\alpha \in \Phi_{+}}\bigsqcup_{M^{\alpha}_{b_{\alpha}} \setminus \{0\}} B(\alpha)$, a morphism of crystals $B(\alpha) \rightarrow \mathcal{C}(M)$ is either $0$ or corresponds to an element of $M^{\alpha}_{b_{\alpha}} \setminus \{0\}$. This correspondence gives our natural isomorphism as required.
\end{proof}

\begin{remark}
As a comodule, $\mathbb{B} \cong \bigoplus_{\alpha} \bigoplus_{b' \in B(-\alpha)} \mathbb{B}(\alpha)$ via $b \otimes b' \mapsto (b)_{b'}$ in the copy of $\mathbb{B}(\alpha)$ indexed by $b' \in B(-\alpha)$. Under this isomorphism, multiplication becomes $(b)_{b'} \cdot (d)_{d'} = (b \cdot d)_{d' \cdot b'}$ whenever this is well defined, and $0$ otherwise, and comultiplication becomes $(b)_{b'} \mapsto \sum_{b'' \in B(\alpha)} (b)_{b''^{\vee}} \otimes (b'')_{b'}$.
\end{remark}

\begin{defn}
A \emph{based $\mathbb{B}$-comodule} is a pair $(M,X)$ such that
\begin{itemize}
\item[i)]$M$ is in $\mathbb{B}\text{-comod}^{\text{free}}$ and $X$ is a free basis of $M$;
\item[ii)]$X = \bigsqcup_{\alpha \in \Phi_{+}} \bigsqcup_{b \in B(\alpha)} X_{b}^{\alpha}$ where $X_{b}^{\alpha} = X \cap M_{b}^{\alpha}$; and
\item[iii)]each $A_{b,b'}^{\alpha}$ restricts to a bijection between the sets $X^{\alpha}_{b} \rightarrow X^{\alpha}_{b'}$.
\end{itemize}
A morphism of based comodules $(M,X) \rightarrow (N,Y)$ is a morphism of comodules $f:M \rightarrow N$ such that $f(X) \subset Y \sqcup \{0\}$. This forms a category which we denote $\mathbb{B}\text{-comod}^{\text{based}}$. The direct sum of two based comodules is $(M,X) \oplus (N,Y) = ({M \oplus N}, X \sqcup Y)$ and their tensor product is $(M,X) \otimes (N,Y) = (M \otimes N, X \otimes Y = \{x \otimes y \mid x \in X, y \in Y \})$.
\end{defn}

\begin{remark}
The data of a basis $X$ in the above definition is equivalent to having chosen a basis $X^{\alpha}_{b_{\alpha}}$ for each $M^{\alpha}_{b_{\alpha}}$ for $\alpha \in \Phi_{+}$.
\end{remark}

\begin{theorem}
\label{CrystalsareBasedComodules}
The functor
$$\mathit{Crys}_{\mathfrak{g}} \rightarrow \mathbb{B}\text{-comod}^{\text{based}}, \quad B(\alpha) \mapsto (\mathbb{B}(\alpha),B(\alpha)),$$
is an equivalence of categories.
\end{theorem}

\begin{proof}
It is clear that $X \mapsto (\mathbb{Z}X,X)$ is functorial. We shall construct a quasi-inverse as follows. Given a based comodule $(M,X)$, the pointed set $X \sqcup \{0\}$ is a subset of $\mathcal{C}(M)$ invariant under the Kashiwara operators, hence naturally forms a subcrystal. Given a morphism of based comodules $(M,X) \rightarrow (M',X')$ the restriction to $X$ maps to $X'$ and commutes with the operators $A_{b,b'}^{\alpha}$, hence with the Kashiwara operators, so gives a morphism of crystals. By the proof of Theorem \ref{B-Comodules are free on crystals}, the composition
$$\mathbb{B}\text{-comod}^{\text{based}} \rightarrow \mathit{Crys}_{\mathfrak{g}} \rightarrow \mathbb{B}\text{-comod}^{\text{based}}$$
is naturally isomorphic to the identity. Likewise, the composition
$$\mathit{Crys}_{\mathfrak{g}} \rightarrow \mathbb{B}\text{-comod}^{\text{based}} \rightarrow \mathit{Crys}_{\mathfrak{g}}$$ is naturally isomorphic to the identity. Hence we have an equivalence.
\end{proof}

\begin{prop}
The functor in Theorem \ref{CrystalsareBasedComodules} gives an equivalence of monoidal categories.
\end{prop}

\begin{proof}
For $\alpha, \beta \in \Phi_+$, the comodule structure of $\mathbb{B}(\alpha) \otimes \mathbb{B}(\beta)$ is
$$b \otimes d \mapsto \sum (b \cdot d) \otimes (d'^{\vee} \cdot b'^{\vee}) \otimes (b' \otimes d')= \sum (b \cdot d) \otimes (b' \cdot d')^{\vee} \otimes (b' \otimes d')$$
where both summations are taken over all $b' \in B(\alpha)$ and $d' \in B(\beta)$ such that $b\otimes d$ and $(d'^{\vee} \otimes b'^{\vee})^{\vee}=b' \otimes d'$ lie in the same connected component, $B(\gamma)$ say. Since all terms of this connected component appear uniquely as some product $b' \cdot d'$, we can then rewrite this as $\sum_{c \in B(\gamma)} (b \cdot d) \otimes c^{\vee} \otimes c$. This is the same comultiplication of $b \otimes d$ as when viewed as an element of $\mathbb{Z}(B(\alpha) \otimes B(\beta))$ under its decomposition into irreducible components. Our result then follows.
\end{proof}

\subsection{Relation to the crystal functor}
Recall from Corollary \ref{CrystalClassificationAsCoalgebras} that $\mathit{Crys}_{\mathfrak{g}}$ is equivalent to the category of coalgebras of the comonad
$$U: \mathit{Set}_{\bullet} \rightarrow \mathit{Set}_{\bullet}, \, X \mapsto \bigsqcup_{\alpha \in \Phi_{+}} \bigsqcup_{\substack{f:FB(\alpha) \rightarrow X \\ f \neq 0}} FB(\alpha).$$

\begin{defn}
For pointed sets $A,B$, we define $\underline{\text{Hom}}_{\mathit{Set}_{\bullet}}(A,B)$ to be pointed set
$$\underline{\text{Hom}}_{\mathit{Set}_{\bullet}}(A,B) = \{f:A \rightarrow B \mid f \neq 0 \} \sqcup \{0:A \rightarrow B\}.$$
\end{defn}

\begin{prop}
The comonad $U$ is isomorphic to
$$U':X \mapsto \bigsqcup_{\alpha \in \Phi_{+}} FB(\alpha) \otimes \underline{\text{Hom}}_{\mathit{Set}_{\bullet}}(FB(\alpha),A).$$
Under this identification, the comultiplication on $U'$, $\Delta:U' \Rightarrow U'U'$, becomes $b \otimes f \mapsto b \otimes f^{\sim}$ where $f^{\sim}(b')=b' \otimes f \in U'(A)$.
\end{prop}

\begin{proof}
The isomorphism is given my the maps
$$FB(\alpha) \otimes \underline{\text{Hom}}_{\mathit{Set}_{\bullet}}(FB(\alpha),A) \rightarrow \bigsqcup_{\substack{f:FB(\alpha) \rightarrow X \\ f \neq 0}} FB(\alpha)$$
taking $b \otimes f$ to $b$ in the copy of $FB(\alpha)$ indexed by $f$.
\end{proof}

\begin{prop}
\label{RewritingTensoringWithB}
The comonad $\mathbb{B} \otimes -$ is isomorphic to
$$A \mapsto \bigoplus_{\alpha \in \Phi_{+}} \mathbb{B}(\alpha) \otimes \underline{\text{Hom}}_{\mathbb{Z}}(\mathbb{B}(\alpha),A).$$
Under this identification, the comultiplication on $\mathbb{B}$ becomes $b \otimes f \mapsto b \otimes f^{\sim}$ where again $f^{\sim}(b')=b' \otimes f$.
\end{prop}

\begin{proof}
For a free abelian group $A$, we have
$$\mathbb{B} \otimes A = \bigoplus_{\alpha} \mathbb{B}(\alpha) \otimes \mathbb{B}(\alpha)^{\vee} \otimes A \cong \bigoplus_{\alpha} \mathbb{B}(\alpha) \otimes \underline{\text{Hom}}_{\mathbb{Z}}(\mathbb{B}(\alpha),A)$$
given by $b \otimes b' \otimes a \mapsto b \otimes [x \mapsto \epsilon(x \otimes b') a]$.
\end{proof}

\begin{remark}
Note that the functor $U'$ clearly does not preserve coproducts, whilst the functor in Proposition \ref{RewritingTensoringWithB} does. As a result, the latter is isomorphic to tensoring with a coalgebra whilst the former is not.
\end{remark}

\subsection{The dual bialgebra}

\begin{defn}
Let
$$\mathbb{B}^{\ast}:=\underline{\text{Hom}}(\mathbb{B},\mathbb{Z}) \cong \prod\nolimits_{\alpha \in \Phi_+} \mathbb{B}(\alpha)^{\ast} \otimes \mathbb{B}(-\alpha)^{\ast}$$
be the dual of $\mathbb{B}$. Let $\{\phi_{b,b'} \mid b \in B(\alpha), b' \in B(-\alpha)\} \subset \mathbb{B}(\alpha)^{\ast} \otimes \mathbb{B}(-\alpha)^{\ast}$ denote the dual $\mathbb{Z}$-basis to $B(\alpha) \otimes B(-\alpha)$, $\phi_{b,b'}(d \otimes d')=\delta_{b,d}\delta_{b',d'}$ for $d \in B(\alpha)$, $d' \in B(-\alpha)$. We shall denote elements of this dual by formal sums $\sum_{b,b'} a_{b,b'} \phi_{b,b'}$ ranging over all $b  \in B(\alpha)$, $b' \in B(-\alpha)$ and $\alpha \in \Phi_+$.
\end{defn}

\begin{lem}
\label{DualBialgebra}
The coalgebra structure on $\mathbb{B}$ induces an algebra structure on the dual $\mathbb{B}^{\ast}$ given by
$$\left( \sum_{b,b'} a_{b,b'} \phi_{b,b'} \right) \cdot \left( \sum_{b,b'} a'_{b,b'} \phi_{b,b'} \right) = \sum_{\substack{\alpha \in \Phi_{+} \\ b \in B(\alpha) \\ b' \in B(-\alpha)}} \left( \sum_{d \in B(\alpha)} a_{b,d^{\vee}} a'_{d,b'} \right) \phi_{b,b'},$$
$$1:= \sum_{\substack{\alpha \in \Phi_{+} \\ b \in B(\alpha)}} \widehat{b \otimes b^{\vee}}.$$
Each $\mathbb{B}(\alpha)$ is a $\mathbb{B}^{\ast}$-module where
$$\phi_{d, d'^{\vee}} \cdot b = \sum_{c \in B(\alpha)} \phi_{d, d'^{\vee}} (b \otimes c^{\vee}) \ c = \delta_{d,b} d'$$
for $b \in B(\alpha)$. Furthermore, the algebra structure on $\mathbb{B}$ induces algebra homomorphisms
$$\Delta:\mathbb{B}^{\ast} \rightarrow (\mathbb{B} \otimes \mathbb{B})^{\ast} \cong \prod\nolimits_{\alpha, \beta \in \Phi_+} \mathbb{B}(\alpha)^{\ast} \otimes \mathbb{B}(-\alpha)^{\ast}\otimes \mathbb{B}(\beta)^{\ast} \otimes \mathbb{B}(-\beta)^{\ast}$$
and $\epsilon:\mathbb{B}^{\ast} \rightarrow \mathbb{Z}$ given by
$$\Delta \left( \sum_{b,b'} a_{b,b'} \phi_{b,b'} \right) = \sum_{\substack{\alpha \in \Phi_{+} \\ b \in B(\alpha) \\ b' \in B(-\alpha)}}\sum_{\substack{\beta \in \Phi_{+} \\ d \in B(\beta) \\ d' \in B(-\beta)}} a_{b\cdot d,d' \cdot b'} \ \phi_{b,b'} \otimes \phi_{d,d'},$$
$$\epsilon \left( \sum_{b,b'} a_{b,b'} \phi_{b,b'} \right) = a_{b_{0},b_{0}^{\vee}},$$
where $B(0)=\{b_0 \}$. Here, the algebra structure on $(\mathbb{B} \otimes \mathbb{B})^{\ast}$ is induced by the coalgebra structure on $\mathbb{B} \otimes \mathbb{B}$.
\end{lem}

\begin{proof}
This is a straightforward verification.
\end{proof}

\begin{defn}
Let $\mathscr{A}$ be a $\mathbb{Z}$-algebra and let $\mathbf{1}_{\mathscr{A}}$ be a subset of $\mathscr{A}$. We will say that $\mathbf{1}_{\mathscr{A}}$ forms a \emph{generalised unit} for $\mathscr{A}$ if, for each $a \in \mathscr{A}$, there is a finite subset $X \subset \mathbf{1}_{\mathscr{A}}$ such that
\begin{itemize}
\item[i)]$a=(\sum_{x \in X} x) \cdot a$; and
\item[ii)]$x \cdot a = 0$ if $x \in \mathbf{1}_{\mathscr{A}} \setminus X$.
\end{itemize}
We will say that an $\mathscr{A}$-module $M$ is \emph{unital} if, for any $m \in M$, there is a finite subset $X \subset \mathbf{1}_{\mathscr{A}}$ such that
\begin{itemize}
\item[i)]$m=(\sum_{x \in X} x) \cdot m$;
\item[ii)]$x \cdot m = 0$ if $x \in \mathbf{1}_{\mathscr{A}} \setminus X$; and additionally
\item[iii)]$x \cdot M$ is a free abelian group for all $x \in \mathbf{1}_{\mathscr{A}}$.
\end{itemize}
\end{defn}

\begin{defn}
Let $\dot{U}_{0}$ be the additive subgroup $\bigoplus_{\alpha \in \Phi_+} \mathbb{B}(\alpha)^{\ast} \otimes \mathbb{B}(-\alpha)^{\ast}$ of $\mathbb{B}^{\ast}$. Let $\mathbf{1}$ denote the collection $\{1_{\alpha} \mid \alpha \in \Phi_{+}\} \subset \dot{U}_{0}$ where $1_{\alpha} = \sum_{b \in B(\alpha)} \phi_{b,b^{\vee}}$.
\end{defn}

\begin{lem}
$\dot{U}_{0}$ is an ideal in $\mathbb{B}^{\ast}$, and hence both a non-unital subalgebra and a $\mathbb{B}^{\ast}$-bimodule. Furthermore, the collection $\mathbf{1}$ forms a generalised unit in $\dot{U}_{0}$.
\end{lem}

\begin{proof}
Given $\sum_{b,b'}a_{b,b'^{\vee}}\phi_{b,b'^{\vee}} \in \mathbb{B}^{\ast}$ and $\phi_{d,d'^{\vee}} \in \mathbb{B}(\alpha)^{\ast} \otimes \mathbb{B}(-\alpha)^{\ast}$ we have
$$\left(\sum_{b,b'}a_{b,b'^{\vee}}\phi_{b,b'^{\vee}}\right) \cdot \phi_{d,d'^{\vee}} = \sum_{b,b'} a_{b,b'^{\vee}}\delta_{b',d} \phi_{b,d'^{\vee}} = \sum_{b \in B(\alpha)} a_{b,d^{\vee}} \phi_{b,d'^{\vee}}$$
which is in $\dot{U}_{0}$. Likewise, $\phi_{d,d'^{\vee}} \cdot (\sum_{b,b'}a_{b,b'^{\vee}}\phi_{b,b'^{\vee}}) \in \dot{U}_{0}$. Furthermore, multiplication by $1_{\alpha}$ is the identity on $\mathbb{B}(\beta)^{\ast} \otimes \mathbb{B}(-\beta)^{\ast}$ when $\beta = \alpha$ and is $0$ otherwise. From this it follows that $\mathbf{1}$ is a generalised unit.
\end{proof}

\begin{remark}
We use the notation $\dot{U}_{0}$ to highlight the similarity with Lusztig's construction of $\dot{U}$ in \cite[p.~183]{L1}. In the following subsection we shall conjecture a more precise relationship between these two constructions.
\end{remark}

\begin{remark}
Note that $\Delta$ from Lemma \ref{DualBialgebra} does not restrict to a comultiplication on $\dot{U}_{0}$. For example, $\Delta(\phi_{b_{0}, b_{0}^{\vee}}) \in (\mathbb{B} \otimes \mathbb{B})^{\ast}$ takes the value $1$ on each $b \otimes b^{\vee} \otimes d \otimes d^{\vee}$ for $b$ a highest weight element and $d$ a lowest weight element in $B(\alpha)$, $\alpha \in \Phi_{+}$, of which there are infinitely many. This is because $b \otimes d$ corresponds to $b_{0}$ in a copy of $B(0)$ under the decomposition of $B(\alpha) \otimes B(\alpha)$ into irreducible components. However, we do have the following collection of maps.
\end{remark}

\begin{defn}
\label{UdotComultiplication}
For $\alpha, \beta, \gamma \in \Phi_{+}$ let
$$\Delta^{\alpha}_{\beta, \gamma}: \mathbb{B}(\alpha)^{\ast} \otimes \mathbb{B}(-\alpha)^{\ast} \rightarrow \mathbb{B}(\beta)^{\ast} \otimes \mathbb{B}(-\beta)^{\ast} \otimes \mathbb{B}(\gamma)^{\ast} \otimes \mathbb{B}(-\gamma)^{\ast},$$
$$\phi_{b, b'} \mapsto \sum_{\substack{\alpha \in \Phi_{+} \\ d \in B(\alpha) \\ d' \in B(-\alpha)}}\sum_{\substack{\beta \in \Phi_{+} \\ d'' \in B(\beta) \\ d''' \in B(-\beta)}} \delta_{b,d\cdot d''}\delta_{b',d''' \cdot d'} \ \phi_{d,d'} \otimes \phi_{d'',d'''}.$$
Let $\varepsilon$ be the restriction of the counit from Lemma \ref{DualBialgebra} to $\dot{U}_{0}$.
\end{defn}

\begin{remark}
The maps in Definition \ref{UdotComultiplication} can be considered as a single map $\dot{U}_{0} \rightarrow (\mathbb{B} \otimes \mathbb{B})^{\ast}$ that agrees with the restriction of the comultiplication in Lemma \ref{DualBialgebra}. They are therefore associative in the sense that
$$(\Delta^{\beta}_{\beta, \delta} \otimes \text{Id}) \circ \Delta^{\alpha}_{\beta, \gamma}= (\text{Id} \otimes \Delta^{\gamma}_{\delta, \gamma}) \circ \Delta^{\alpha}_{\beta, \gamma}$$
for all $\alpha,\beta,\gamma, \delta \in \Phi_{+}$.
\end{remark}

\begin{prop}
The maps $\Delta^{\alpha}_{\beta,\gamma}$ and $\varepsilon$ induce a monoidal structure on the category of unital $\dot{U}_{0}$-modules where $x \in \mathbb{B}(\alpha)^{\ast} \otimes \mathbb{B}(-\alpha)^{\ast}$ acts on $(\mathbf{1}_{\beta} \otimes \mathbf{1}_{\gamma})(M \otimes N)$ as $\Delta^{\alpha}_{\beta,\gamma}(x)$ for unital modules $M$ and $N$.
\end{prop}

\begin{proof}
We first note that $\mathbb{B}(\alpha)^{\ast} \otimes \mathbb{B}(-\alpha)^{\ast}$ are unital subalgebras of $\dot{U}_{0}$ with unit $1_{\alpha}$, and that $\Delta^{\alpha}_{\beta,\gamma}$ are algebra homomorphisms. Also, since $M = \bigoplus_{\alpha \in \Phi_{+}} \mathbf{1}_{\alpha} M$ and $N = \bigoplus_{\beta \in \Phi_{+}} \mathbf{1}_{\beta} N$, we obtain a well defined action of $\dot{U}_{0}$ on $M \otimes N$. The associativity constraint observed in the previous remark ensures that the monoidal structure in associative. Furthermore, the fact that the compositions
$$\mathbb{B}(\alpha)^{\ast} \otimes \mathbb{B}(-\alpha)^{\ast} \xrightarrow{\Delta^{\alpha}_{\beta,\gamma}} \mathbb{B}(\beta)^{\ast} \otimes \mathbb{B}(-\beta)^{\ast} \otimes \mathbb{B}(\gamma)^{\ast} \otimes \mathbb{B}(-\gamma)^{\ast} \xrightarrow{\text{Id} \otimes \varepsilon} \mathbb{B}(\beta)^{\ast} \otimes \mathbb{B}(-\beta)^{\ast}$$
and
$$\mathbb{B}(\alpha)^{\ast} \otimes \mathbb{B}(-\alpha)^{\ast} \xrightarrow{\Delta^{\alpha}_{\beta,\gamma}} \mathbb{B}(\beta)^{\ast} \otimes \mathbb{B}(-\beta)^{\ast} \otimes \mathbb{B}(\gamma)^{\ast} \otimes \mathbb{B}(-\gamma)^{\ast} \xrightarrow{\varepsilon \otimes \text{Id}} \mathbb{B}(\gamma)^{\ast} \otimes \mathbb{B}(-\gamma)^{\ast}$$
are the identity when $\alpha=\beta$ or $\alpha = \gamma$ respectively or $0$ otherwise ensures that $\mathbb{Z}$ with the action of $\dot{U}_{0}$ given by $\varepsilon$ is a monoidal unit.
\end{proof}

\begin{defn}
A \emph{based $\dot{U}_{0}$-module} is a pair $(M,X)$ such that
\begin{itemize}
\item[i)]$M$ is a $\dot{U}_{0}$-module and $X$ is a free $\mathbb{Z}$-basis of $M$;
\item[ii)]$X = \bigsqcup_{\alpha \in \Phi_{+}} \bigsqcup_{b \in B(\alpha)} X_{b}^{\alpha}$ where $X_{b}^{\alpha} := X \cap \phi_{b,b^{\vee}} \cdot M$ for $b \in B(\alpha)$; and
\item[iii)]the action of each $\phi_{b,b'^{\vee}}$ restricts to a bijection between the sets $X^{\alpha}_{b} \rightarrow X^{\alpha}_{b'}$.
\end{itemize}
A morphism of based comodules module homomorphism that preserves the free basis. The tensor product of two based modules is $(M,X) \otimes (N,Y) = (M \otimes N, X \otimes Y = \{x \otimes y \mid x \in X, y \in Y \})$.
\end{defn}

\begin{prop}
\label{UdotModulesEquivalentToCrystals}
The monoidal category of based unital $\dot{U}_{0}$-modules is equivalent to $\mathit{Crys}_{\mathfrak{g}}$.
\end{prop}

\begin{proof}
By Theorem \ref{CrystalsareBasedComodules}, it is enough to show that based unital $\dot{U}_{0}$-modules are equivalent to based $\mathbb{B}$-comodules.\\

Given a based unital $\dot{U}_{0}$-module $(M,X)$, the maps
$$M \rightarrow \mathbb{B}(\alpha) \otimes \mathbb{B}(-\alpha) \otimes M, \quad m \mapsto \sum_{b,b' \in B(\alpha)} b \otimes b'^{\vee} \otimes (\phi_{b,b'^{\vee}}m),$$
are non-zero for only finitely many $\alpha \in \Phi_{+}$. Summing these gives a coaction of $\mathbb{B}$ on $M$. By construction this makes $(M,X)$ a based $\mathbb{B}$-comodule.\\

Conversely, given a based $\mathbb{B}$-comodule $(M,X)$ we obtain a $\dot{U}_{0}$-module structure via the composition
$$\dot{U}_{0} \otimes M \xrightarrow{\Delta_{M}} \dot{U}_{0} \otimes \mathbb{B} \otimes M \xrightarrow{\langle-,-\rangle \otimes \text{Id}} \mathbb{Z} \otimes M \cong M.$$
For each $m \in M$, the coaction $\Delta_{M}(m)$ is a finite sum in the free basis $\mathcal{B} \otimes X$. Hence there is a finite subset $A \subset \Phi_{+}$ such that the non-zero coefficients are for basis elements in $(\bigsqcup_{\alpha \in A} B(\alpha) \otimes B(-\alpha)) \otimes X \subset \mathcal{B} \otimes X$. Thus $(\sum_{\alpha \in A} 1_{\alpha})m=m$ under this action of $\dot{U}_{0}$ and $1_{\beta} \cdot m = 0$ for $\beta \not\in A$. So $(M,X)$ becomes a based unital $\dot{U}_{0}$-modules.\\

These correspondences give mutually inverse functors from which our stated equivalence arises. This equivalence is monoidal since the monoidal structure on $\dot{U}_{0}$-modules is induced by the comultiplication maps $\Delta^{\alpha}_{\beta,\gamma}$ which collectively are dual to the multiplication on $\mathbb{B}$ which induces the monoidal structure on $\mathbb{B}$-comodules.
\end{proof}

\begin{defn}
For $i \in I$ and $\alpha \in \Phi_{+}$, let
$$\begin{array}{rclrcl}
\tilde{e}_{\alpha, i} &=& \sum_{\substack{b \in B(\alpha)\\ \tilde{e}_{i}b \neq 0}} \phi_{\tilde{e}_{i}b, b^{\vee}}, & \tilde{f}_{\alpha, i} &=& \sum_{\substack{b \in B(\alpha)\\ \tilde{f}_{i}b \neq 0}} \phi_{\tilde{f}_{i}b, b^{\vee}},\\
\end{array}$$
in $\dot{U}_{0}$.
\end{defn}

\begin{lem}
For all $i \in I$ and $\alpha \in \Phi_{+}$, $\tilde{e}_{\alpha, i}$ and $\tilde{f}_{\alpha, i}$ act as $\tilde{e}_{i}$ and $\tilde{f}_{i}$ on $\mathbb{B}(\alpha)$, and by zero on $\mathbb{B}(\beta)$ for $\beta \neq \alpha$.
\end{lem}

\begin{proof}
Recall that for $b_{0} \in B(\beta)$, $\tilde{e}_{\alpha, i}$ acts as
$$\sum_{\substack{b \in B(\alpha)\\ \tilde{e}_{i}b \neq 0}} \phi_{\tilde{e}_{i}b_{0}, b^{\vee}} \cdot b  = \sum_{\substack{b \in B(\alpha)\\ \tilde{e}_{i}b \neq 0}} \delta_{\tilde{e}_{i}b_{0},b} b$$
which is $\tilde{e}_{i}b_{0}$ if $\alpha=\beta$ and $0$ otherwise. Likewise for the action of $\tilde{f}_{\alpha, i}$.
\end{proof}

\begin{prop}
$\dot{U}_{0}$ is generated as an algebra by
$$\{\tilde{e}_{\alpha, i}, \tilde{f}_{\alpha, i} \mid i \in I, \alpha \in \Phi_{+}\}$$
along with the generalised unit elements $\mathbf{1}=\{1_{\alpha} \mid \alpha \in \Phi_{+}\}$.
\end{prop}

\begin{proof}
Fix $\alpha \in \Phi_{+}$. For $i \in I$, $1_{\alpha}-\tilde{f}_{\alpha,i}\tilde{e}_{\alpha,i}$ is the sum of $\phi_{b, b^{\vee}}$ such that $\tilde{e}_{i}b=0$. So, for any ordering of $I$, $\prod_{i \in I} (1_{\alpha}-\tilde{f}_{\alpha,i}\tilde{e}_{\alpha,i})$ is the sum of $\phi_{b, b^{\vee}}$ where $\tilde{e}_{i}b=0$ for all $i \in I$. That is,
$$\phi_{b_{\alpha}, b_{\alpha}^{\vee}} = \prod_{i \in I} (1_{\alpha}-\tilde{f}_{\alpha,i}\tilde{e}_{\alpha,i}).$$
The result then follows from the fact that if $b=\tilde{f}_{i_{1}}\tilde{f}_{i_{2}}..\tilde{f}_{i_{n}}b_{\alpha}$ and $b'=\tilde{f}_{j_{1}}\tilde{f}_{j_{2}}..\tilde{f}_{j_{m}}b_{\alpha}$ then
$$
\begin{array}{rcl}
\phi_{b, b'^{\vee}} &=& \tilde{f}_{\alpha,i_{1}}\tilde{f}_{\alpha,i_{2}}..\tilde{f}_{\alpha,i_{n}}(\phi_{b_{\alpha}, b_{\alpha}^{\vee}})\tilde{e}_{\alpha,j_{1}}\tilde{e}_{\alpha,j_{2}}..\tilde{e}_{\alpha,j_{m}}\\
&=& \tilde{f}_{\alpha,i_{1}}..\tilde{f}_{\alpha,i_{n}}\left(\prod_{i \in I} (1_{\alpha}-\tilde{f}_{\alpha,i}\tilde{e}_{\alpha,i})\right)\tilde{e}_{\alpha,j_{1}}..\tilde{e}_{\alpha,j_{m}}.
\end{array}
$$
\end{proof}

\subsection{Relation to global bases}

Kashiwara shows in \cite{K6} that crystal bases $B(\alpha)$ of representations $V(\alpha)$ induce global bases of the vector spaces $V(\alpha)$. Using these bases, we see that $\mathcal{B}$ gives rise to a global base of $A_{q}(\mathfrak{g})$.
\\

Recall from Proposition \ref{BPresentationForsl2} that, in the case of $\mathfrak{sl}_{2}$, the bialgebra $\mathbb{B}$ is generated by $a,b,c,d$ which satisfy the relations
$$cb=bc=db=dc=ba=ca = 0, \, \, \, da=1.,$$
This forms a bialgebra via the comultiplication
$$\begin{array}{cc}
\Delta(a)=a \otimes a + b \otimes c, & \Delta(b)=a \otimes b + b \otimes d,\\
\Delta(c)= c \otimes a + d \otimes c, & \Delta(d)=c \otimes b + d \otimes d.
\end{array}$$
Similarly, the quantum coordiante ring $A_{q}(\mathfrak{sl}_{2})$ can be realised as a quotient of the fee algebra $k\langle a,b,c,d \rangle$ by the relations
$$\begin{array}{ccc}
cb=bc=qad-q1, & db=qbd, & dc=qcd,\\
ba=qab, & ca=q  ac, & da=qcb+1,
\end{array}$$
again viewed as a bialgebra with an analogous comultiplication to the above. Kashiwara shows in \cite{K5} that
$$\{a^{i}c^{j}d^{k} \mid i,j,k \geq 0\} \cup \{a^{i}b^{j}d^{k} \mid i,j,k \geq 0, j \neq 0\} \subset A_{q}(\mathfrak{sl}_{2})$$
is the global basis of $A_{q}(\mathfrak{sl}_{2})$ corresponding to the cystal base $\mathcal{B}$, where $a^{i}c^{j}d^{k} = u_{i+j}^{(i+j+k)} \otimes (u_{i}^{(j+k)})^{\vee}$ and $a^{i}b^{j}d^{k} = u_{i}^{(j+k)} \otimes (x_{i+j}^{(k)})^{\vee}$ in $\mathcal{B}$. If we assume $k=K(q)$ for some field $K$, so $q$ is a formal parameter, it is then apparent that the multiplication in $\mathbb{B}$ on basis elements is the result of taking the corresponding global basis elements in $A_{q}(\mathfrak{sl}_{2})$, writing their product again in terms of global basis elements and then taking only the $q^{0}$ coefficient (that is, evaluating at $q=0$). A similar process can be formulated for the comultiplication.
\\

It is a goal of future work by the author to investigate whether this phenomenon is exclusive to $\mathfrak{sl}_{2}$. In \cite{L3}, Lusztig uses a similar process of multiplying global basis elements (or \emph{canonical basis} elements in his terminology) of a modified version of $U_{q}(\mathfrak{g})$, denoted $\dot{U}$, to construct a bialgebra. He refers to his construction as a quantum group at $v=\infty$, or $q=v^{-1}=0$ in our notation. Since $\dot{U}$ is dual to $A_{q}(\mathfrak{g})$, this bialgebra at $q=0$ should be dual to $\mathbb{B}$. We conjecture that it is $\dot{U}_{0}$. This should give some way of describing the (co)multiplication of $\mathbb{B}$ in terms of the (co)multiplication of global basis elements of $A_{q}(\mathfrak{g})$.

\end{document}